\documentclass[a4paper,reqno,10pt]{article}

\usepackage{epsf,graphicx,epsfig}
\usepackage{amscd}
\usepackage{amsmath,latexsym,amssymb,amsthm}
\usepackage[nospace,noadjust]{cite}
\usepackage{textcomp}
\usepackage{setspace,cite}
\usepackage{lscape,fancyhdr,fancybox}
\usepackage{stmaryrd}
\usepackage[all,cmtip]{xy}
\usepackage{tikz}
\usepackage{bbm}
\usepackage{cancel}
\usetikzlibrary{shapes,arrows,decorations.markings}
\usepackage{graphicx}

\usepackage{color}
\usepackage{url}
\usepackage{enumerate}
\usepackage[mathscr]{euscript}

  \theoremstyle{plain}

\swapnumbers
    \newtheorem{thm}{Theorem}[section]
    \newtheorem{prop}[thm]{Proposition}

    \newtheorem{subsec}[thm]{}
\theoremstyle{definition}
    \newtheorem{defn}[thm]{Definition}
        \newtheorem{remark}[thm]{Remark}
    \newtheorem{exam}[thm]{Example}

\theoremstyle{remark}

\title{}
\author{}
\date{}
\usepackage{amssymb}

\usepackage{hyperref}
\hypersetup{
	colorlinks,
	citecolor=blue,
	filecolor=black,
	linkcolor=blue,
	urlcolor=black
}

\begin{document}

\title{{\sf Cohomology of weighted Rota-Baxter Lie algebras and Rota-Baxter paired operators}}

\author{\em Apurba Das}

\maketitle

\begin{center}
{\em Department of Mathematics and Statistics,\\
Indian Institute of Technology, Kanpur 208016, Uttar Pradesh, India.\\
Email: apurbadas348@gmail.com}
\end{center}



\medskip

\medskip

\begin{abstract}
{In this paper, we define representations and cohomology of weighted Rota-Baxter Lie algebras. As applications of cohomology, we study abelian extensions and formal $1$-parameter deformations weighted Rota-Baxter Lie algebras. Finally, we consider weighted Rota-Baxter paired operators that induces a weighted Rota-Baxter Lie algebra and a representation of it. We define suitable cohomology for such paired operators that govern deformation.}
\end{abstract}

\medskip



\medskip


\thispagestyle{empty}

\tableofcontents

\vspace{0.2cm}

\section{Introduction}

Rota-Baxter operators were first appeared in the work of Baxter in his study of the fluctuation theory \cite{baxter} and further developed by Rota in combinatorics \cite{rota}. Subsequently, Cartier \cite{cartier} and Atkinson \cite{atkinson} studied further properties of Rota-Baxter operators. In last twenty years, Rota-Baxter operators on associative algebras pay very much attention due to its connection with combinatorics of shuffle algebras \cite{guo-kei}, Yang-Baxter equations \cite{aguiar}, dendriform algebras \cite{aguiar,fard-guo}, renormalizations in quantum field theory \cite{connes}, multiple zeta values in number theory \cite{guo-zeta} and splitting of algebraic operads \cite{bai-spl}. See \cite{guo-book} for more details about Rota-Baxter operators on associative algebras. On the other hand, Rota-Baxter operators on Lie algebras was first considered by Kuperscmidt in the study of classical $r$-matrices \cite{kuper}. They are also connected with pre-Lie algebras \cite{bai-spl}, integrable systems \cite{semenov} and combinatorics of Lyndon-Shirshov words \cite{qiu}. 

\medskip

Deformation theory of some algebraic structure goes back to Gerstenhaber \cite{gers} for associative algebras and Nijenhuis-Richardson \cite{nij-ric} for Lie algebras. Recently, deformation theory has been adapted to Rota-Baxter operators on Lie algebras \cite{tang} and subsequently developed in \cite{das-rota} for associative algebras. In both these papers, the authors only considered deformations of Rota-Baxter operators by keeping the underlying algebras intact. In \cite{lazarev,das-mishra} the authors dealt with deformations of Rota-Baxter algebras in which they simultaneously deform Rota-Baxter operators and the underlying algebras. Note that all these works are concern about Rota-Baxter operators of weight zero.

\medskip

Rota-Baxter operators with arbitrary weight (also called weighted Rota-Baxter operators) was considered in \cite{bai-w,guo-lax}. They are related with tridendriform algebras \cite{fard}, post-Lie algebras and modified Yang-Baxter equations \cite{guo-lax}, weighted infinitesimal bialgebras and weighted Yang-Baxter equations \cite{aybe}, combinatorics of rooted forests \cite{forest}, among others. Recently, the authors in \cite{sheng-w1} defined the cohomology of Rota-Baxter operators of weight $1$ on Lie algebras and Lie groups. This motivates the present author to study cohomology and deformations of weighted Rota-Baxter operators on both associative and Lie algebras \cite{das-weighted}. However, the simultaneous deformations of (associative) algebras and weighted Rota-Baxter operators are considered in a recent preprint of Wang and Zhou \cite{zhou}. More precisely, they considered weighted Rota-Baxter associative algebras and define cohomology of them with coefficients in a suitable Rota-Baxter bimodule. When considering the cohomology with coefficients in itself, it governs the simultaneous deformations of algebras and weighted Rota-Baxter operators.

\medskip

Our aim in this paper is to apply the approach of \cite{zhou} to weighted Rota-Baxter Lie algebras. More precisely, we first consider representations of weighted Rota-Baxter Lie algebras and provide various constructions of representations. Then we define the cohomology of a weighted Rota-Baxter Lie algebra with coefficients in a representation. This cohomology is obtained as a byproduct of the standard Chevalley-Eilenberg cohomology of the underlying Lie algebra and the cohomology of the underlying weighted Rota-Baxter operator. When the weight is zero, our cohomology coincides with the one introduced in \cite{lazarev}. We interpret our second cohomology group as the isomorphism classes of abelian extensions of weighted Rota-Baxter Lie algebras. Then we consider deformations of weighted Rota-Baxter Lie algebras in which we simultaneously deform underlying Lie algebras and weighted Rota-Baxter operators. The infinitesimals of such deformations are $2$-cocycles in the cohomology and equivalent deformations produce cohomologous $2$-cocycles. Hence they correspond to the same element in the second cohomology group. We also find a sufficient condition for the rigidity of a weighted Rota-Baxter Lie algebra.

\medskip

Finally, given a Lie algebra and a representation, we introduce Rota-Baxter paired operators that induces a weighted Rota-Baxter Lie algebra and representation of it. The terminology of Rota-Baxter paired operators is motivated by a paper by Zheng, Guo and Zhang \cite{guo-rbp} where the authors introduced Rota-Baxter paired modules in the associative context. We define a suitable differential graded Lie algebra that characterize Rota-Baxter paired operators as its Maurer-Cartan elements. This suggests us to define the cohomology of Rota-Baxter paired operators that control formal deformations.

\medskip

The paper is organized as follows. In Section \ref{sec-2} we consider weighted Rota-Baxter Lie algebras and their representations. We also give several new constructions of representations. The cohomology of a weighted Rota-Baxter Lie algebra with coefficients in a representation is defined in Section \ref{sec-3}. Applications of such cohomology to abelian extensions and formal $1$-parameter deformations are given in Section \ref{sec-4}. Finally, in Section \ref{sec-rbp}, we define Rota-Baxter paired operators. We provide the cohomology of such paired operators and consider their deformations.

\medskip

All vector spaces, (multi)linear maps, wedge products are over a field {\bf k} of characteristic $0$.

\section{Representations of weighted Rota-Baxter Lie algebras}\label{sec-2}
In this section, we consider weighted Rota-Baxter Lie algebras \cite{tang-2} and introduce their representations. We also provide various examples and new constructions. Let $\lambda \in {\bf k}$ be a fixed scalar unless specified otherwise.


\begin{defn}
(i) Let $\mathfrak{g}$ be a Lie algebra. A linear map $\mathfrak{T} : \mathfrak{g} \rightarrow \mathfrak{g}$ is said to be a $\lambda$-weighted Rota-Baxter operator if $\mathfrak{T}$ satisfies
\begin{align}\label{lam-rota-id}
[ \mathfrak{T}(x), \mathfrak{T}(y)] = \mathfrak{T} \big( [\mathfrak{T}(x), y] + [ x, \mathfrak{T} (y)] + \lambda [x, y]   \big), ~\text{ for } x, y \in \mathfrak{g}.
\end{align}

(ii) A $\lambda$-weighted Rota-Baxter Lie algebra is a pair $(\mathfrak{g}, \mathfrak{T})$ consisting of a Lie algebra $\mathfrak{g}$ together with a $\lambda$-weighted Rota-Baxter operator on it.
\end{defn}

\medskip

\begin{exam}
(i) For any Lie algebra $\mathfrak{g}$, the pair $(\mathfrak{g}, \mathrm{id}_\mathfrak{g})$ is a $(-1)$-weighted Rota-Baxter Lie algebra.

\medskip

(ii) Let $(\mathfrak{g}, \mathfrak{T})$ be a $\lambda$-weighted Rota-Baxter Lie algebra. Then for any $\mu \in {\bf k}$, the pair $(\mathfrak{g}, \mu \mathfrak{T})$ is a $(\mu \lambda)$-weighted Rota-Baxter Lie algebra.

\medskip

(iii) Let $(\mathfrak{g}, \mathfrak{T})$ be a $\lambda$-weighted Rota-Baxter Lie algebra. Then $(\mathfrak{g}, -\lambda \mathrm{id}_\mathfrak{g} - \mathfrak{T})$ is so.

\medskip

(iv) Given a $\lambda$-weighted Rota-Baxter Lie algebra $(\mathfrak{g}, \mathfrak{T})$ and an automorphism $\psi \in \mathrm{Aut}(\mathfrak{g})$ of the Lie algebra $\mathfrak{g}$, the pair $(\mathfrak{g}, \psi^{-1} \circ \mathfrak{T} \circ \psi)$ is a $\lambda$-weighted Rota-Baxter Lie algebra.

\medskip

(v) Let $\mathfrak{g}$ be a Lie algebra which splits as the direct sum of two subalgebras $\mathfrak{g}_{-}$ and $\mathfrak{g}_{+}$. Then $(\mathfrak{g}, \mathfrak{T})$ is a $\lambda$-weighted Rota-Baxter Lie algebra, where $\mathfrak{T}: \mathfrak{g} \rightarrow \mathfrak{g}$ is given by
\begin{align*}
\mathfrak{T}(x_{-}, x_{+}) = (0, - \lambda x_{+}), ~ \text{ for } ~ (x_{-}, x_{+}) \in \mathfrak{g}_{-} \oplus \mathfrak{g}_{+} = \mathfrak{g}.
\end{align*}

\medskip

(vi) This example generalizes the previous one. Let $\mathfrak{g}$ be a Lie algebra which splits as the direct sum of three subalgebras $\mathfrak{g}_{-}$, $\mathfrak{g}_0$ and $\mathfrak{g}_{+}$ in which $\mathfrak{g}_{-}$ and $\mathfrak{g}_{+}$ are both $\mathfrak{g}_0$-modules (i.e., representations of $\mathfrak{g}_0$). If $(\mathfrak{g}_0, \mathfrak{T}_0)$ is a $\lambda$-weighted Rota-Baxter Lie algebra then $(\mathfrak{g}, \mathfrak{T})$ is so, where $\mathfrak{T}: \mathfrak{g} \rightarrow \mathfrak{g}$ is given by
\begin{align*}
\mathfrak{T} (x_-, x_0, x_+) = (0, \mathfrak{T}_0 (x_0), - \lambda x_+), ~ \text{ for } (x_-, x_0, x_+) \in \mathfrak{g}.
\end{align*}
\end{exam}

\begin{defn}
Let $(\mathfrak{g}, \mathfrak{T})$ and $(\mathfrak{g}', \mathfrak{T}')$ be two $\lambda$-weighted Rota-Baxter Lie algebras. A morphism from  $(\mathfrak{g}, \mathfrak{T})$ to $(\mathfrak{g}', \mathfrak{T}')$ is a Lie algebra homomorphism $\phi : \mathfrak{g} \rightarrow \mathfrak{g}'$ satisfying additionally $\phi \circ \mathfrak{T} = \mathfrak{T}' \circ \phi.$ It is called an isomorphism if $\phi$ is so.
\end{defn}

Let $\mathfrak{g}$ be a Lie algebra. 
Recall that a representation of $\mathfrak{g}$ is a vector space $\mathcal{V}$ with a linear map (called the action map) $\rho : \mathfrak{g} \rightarrow \mathrm{End}(\mathcal{V})$ satisfying
\begin{align*}
\rho [x,y] =  \rho (x) \circ \rho (y) - \rho (y) \circ \rho (x), ~ \text{ for } x, y \in \mathfrak{g}.
\end{align*}
We denote a representation of $\mathfrak{g}$ as above by $(\mathcal{V}, \rho)$ or simply by $\mathcal{V}$ when no confusion arises. Note that any Lie algebra $\mathfrak{g}$ is a representation of itself with the action map $\rho : \mathfrak{g} \rightarrow \mathrm{End}(\mathfrak{g})$ given by $\rho (x)(y) = [x, y]$, for $x, y \in \mathfrak{g}$. This is called the adjoint representation.

\begin{defn}
Let $(\mathfrak{g}, \mathfrak{T})$ be a $\lambda$-weighted Rota-Baxter Lie algebra. A representation of it is a pair $(\mathcal{V}, \mathcal{T})$ in which $\mathcal{V} = (\mathcal{V}, \rho)$ is a representation of the Lie algebra $\mathfrak{g}$ and $\mathcal{T} : \mathcal{V} \rightarrow \mathcal{V}$ is a linear map satisfying
\begin{align}\label{lam-rep}
\rho (Tx) (\mathcal{T} u) = \mathcal{T} \big( \rho (\mathfrak{T}x) u + \rho (x) (\mathcal{T}u) + \lambda \rho (x) u  \big), ~\text{ for } x \in \mathfrak{g}, u \in \mathcal{V}.
\end{align}
\end{defn}

\begin{exam}
Any $\lambda$-weighted Rota-Baxter Lie algebra $(\mathfrak{g}, \mathfrak{T})$ is a representation of itself. We call this the adjoint representation.
\end{exam}

\begin{exam}
Let $\mathfrak{g}$ be a Lie algebra and $\mathcal{V}$ be a representation of it. Then the pair $(\mathcal{V}, \mathrm{id}_\mathcal{V})$ is a representation of the $(-1)$-weighted Rota-Baxter Lie algebra $(\mathfrak{g}, \mathrm{id}_\mathfrak{g})$.
\end{exam}

\begin{exam}
Let $(\mathfrak{g}, \mathfrak{T})$ be a $\lambda$-weighted Rota-Baxter Lie algebra and $(\mathcal{V}, \mathcal{T})$ be a representation of it. Then for any scalar $\mu \in {\bf k}$, the pair $(\mathcal{V}, \mu \mathcal{T})$ is a representation of the $(\mu \lambda)$-weighted Rota-Baxter Lie algebra $(\mathfrak{g}, \mu \mathfrak{T})$.
\end{exam}


\begin{prop}
Let $(\mathfrak{g}, \mathfrak{T})$ be a $\lambda$-weighted Rota-Baxter Lie algebra and $(\mathcal{V}, \mathcal{T})$ be a representation of it. Then $(\mathcal{V}, - \lambda \mathrm{id}_\mathcal{V} - \mathcal{T})$ is a representation of the $\lambda$-weighted Rota-Baxter Lie algebra $(\mathfrak{g}, - \lambda \mathrm{id}_\mathfrak{g} - \mathfrak{T})$.
\end{prop}

\begin{proof}
For any $x \in \mathfrak{g}$ and $u \in \mathcal{V}$, we observe that
\begin{align}\label{minus-1}
&\rho ((- \lambda \mathrm{id}_\mathfrak{g} - \mathfrak{T})x ) (- \lambda \mathrm{id}_\mathcal{V} - \mathcal{T})(u) \nonumber \\
&= \lambda^2 \rho(x) u + \lambda \rho (\mathfrak{T}x) u + \lambda \rho (x) \mathcal{T}u + \rho (\mathfrak{T}x) \mathcal{T}u \nonumber \\
&= \lambda^2 \rho(x) u + \lambda \rho (\mathfrak{T}x) u + \lambda \rho (x) \mathcal{T}u + \mathcal{T} \big( \rho(\mathfrak{T}x) u + \rho (x) \mathcal{T}u + \lambda \rho (x) u   \big).
\end{align}
On the other hand,
\begin{align}\label{minus-2}
&(- \lambda \mathrm{id}_\mathcal{V} - \mathcal{T}) \bigg( \rho ( - \lambda \mathrm{id}_\mathfrak{g} - \mathfrak{T}(x) ) u + \rho (x) (- \lambda \mathrm{id}_\mathcal{V} - \mathcal{T})(u) + \lambda \rho (x) u    \bigg) \nonumber \\
&= \lambda^2 \rho (x) u + \lambda \mathcal{T} (\rho (x) u) + \lambda \rho (\mathfrak{T}x) u + \mathcal{T} (\rho (\mathfrak{T}x) u)
+ \cancel{\lambda^2 \rho (x) u} + \cancel{\lambda \mathcal{T} (\rho (x) u)} + \lambda \rho (x) \mathcal{T}u + \mathcal{T} (\rho (x) \mathcal{T}u) \\
&~~ - \cancel{\lambda^2 \rho (x) u} - \cancel{\lambda \mathcal{T} (\rho(x) u)}. \nonumber
\end{align}
The expressions in (\ref{minus-1}) and (\ref{minus-2}) are same. Hence $(\mathcal{V}, - \lambda \mathrm{id}_\mathcal{V} - \mathcal{T})$ is a representation of the $\lambda$-weighted Rota-Baxter Lie algebra $(\mathfrak{g}, - \lambda \mathrm{id}_\mathfrak{g} - \mathfrak{T})$.
\end{proof}

\begin{prop}
Let $(\mathfrak{g}, \mathfrak{T})$ be a $\lambda$-weighted Rota-Baxter Lie algebra and $\{ (\mathcal{V}_i, \mathcal{T}_i) \}_{i \in I}$ be a family of representations of it. Then $\big( \oplus_{i \in I} \mathcal{V}_i, ~ \oplus_{i \in I} \mathcal{T}_i   \big)$ is also a representation of the $\lambda$-weighted Rota-Baxter Lie algebra $(\mathfrak{g}, \mathfrak{T})$.
\end{prop}

\begin{proof}
Let $\rho_i : \mathfrak{g} \rightarrow \mathrm{End}(\mathcal{V}_i)$ denote the action of the Lie algebra $\mathfrak{g}$ on the representation $\mathcal{V}_i$. Then it follows that $\rho : \mathfrak{g} \rightarrow \mathrm{End} (\oplus_{i \in I} \mathcal{V}_i)$, $\rho (x) (u_i)_{i \in I} = (\rho_i (x) u_i)_{i \in I}$ is a representation of $\mathfrak{g}$ on $\oplus_{i \in I} \mathcal{V}_i$. Moreover, for any $x \in \mathfrak{g}$ and $(u_i)_{i \in I} \in \oplus_{i \in I} \mathcal{V}_i$, we have
\begin{align*}
\rho (\mathfrak{T}x)     (\oplus_{i \in I} \mathcal{T}_i) (u_i)_{i \in I} =~& \big(   \rho_i (\mathfrak{T}x) \mathcal{T}_i (u_i) \big)_{i \in I} \\
=~& \big(   \mathcal{T}_i \big( \rho_i (\mathfrak{T}x) u_i + \rho_i (x) \mathcal{T}_i (u_i) + \lambda \rho_i(x) u_i    \big)     \big)_{i \in I} \\
=~&  (\oplus_{i \in I} \mathcal{T}_i)  \bigg(  \rho (\mathfrak{T}x)(u_i)_{i \in I} + \rho (x)  (\oplus_{i \in I} \mathcal{T}_i)  (u_i)_{i \in I} + \lambda \rho (x)(u_i)_{i \in I}   \bigg).
\end{align*} 
Hence the result follows.
\end{proof}


Let $\mathfrak{g}$ be a Lie algebra and $(\mathcal{V}, \rho)$ be a representation. Then there is a representation of the Lie algebra $\mathfrak{g}$ on the space $\mathrm{End}(\mathcal{V})$ with the action given by
\begin{align*}
\widetilde{\rho} : \mathfrak{g} \rightarrow \mathrm{End} \big(  \mathrm{End}(\mathcal{V}) \big), ~ (\widetilde{\rho} (x) f) u = - f (\rho (x) u), ~ \text{ for } x \in \mathfrak{g},~ f \in \mathrm{End}(\mathcal{V}) \text{ and } u \in \mathcal{V}.
\end{align*}
With this representation, we have the following.

\begin{prop}
Let $(\mathfrak{g}, \mathfrak{T})$ be a $\lambda$-weighted Rota-Baxter Lie algebra and $(\mathcal{V}, \mathcal{T})$ be a representation of it. Then $(\mathrm{End}(\mathcal{V}), \widetilde{\mathcal{T}})$ is also a representation of the $\lambda$-weighted Rota-Baxter Lie algebra $(\mathfrak{g}, \mathfrak{T})$, where
\begin{align*}
\widetilde{\mathcal{T}} : \mathrm{End}(\mathcal{V}) \rightarrow \mathrm{End}(\mathcal{V}), ~ \widetilde{\mathcal{T}} (f) (u) = - \lambda f (u) - f (\mathcal{T}(u)), ~\text{ for } u \in V.
\end{align*}
\end{prop}

\begin{proof}
For any $x, y \in \mathfrak{g}$, $f \in \mathrm{End}(\mathcal{V})$ and $u \in \mathcal{V}$,
\begin{align}\label{end-1}
\bigg( \widetilde{\rho} (\mathcal{T} x) \widetilde{T}(f)    \bigg) u 
= - \widetilde{\mathcal{T}} (f) \big(   \rho (\mathfrak{T}x) u \big) 
= \lambda ~f  (\rho (\mathfrak{T}x) u) + f \big(  \mathcal{T} (\rho (\mathfrak{T}x)u) \big).
\end{align}
On the other hand,
\begin{align}\label{end-2}
&\bigg( \widetilde{\mathcal{T}} \big(  \widetilde{\rho}(\mathfrak{T} x) f + \widetilde{\rho}(x) \widetilde{\mathcal{T}} (f) + \lambda \widetilde{\rho}(x) f   \big)   \bigg)u \nonumber \\
&= - \lambda \big(  \widetilde{\rho}(\mathfrak{T} x) f + \widetilde{\rho}(x) \widetilde{\mathcal{T}} (f) + \lambda \widetilde{\rho}(x) f \big)  u ~-~ \big(  \widetilde{\rho}(\mathfrak{T} x) f + \widetilde{\rho}(x) \widetilde{\mathcal{T}} (f) + \lambda \widetilde{\rho}(x) f \big) (\mathcal{T} u)  \nonumber \\
&= \lambda \bigg(  f (\rho (\mathfrak{T}x)u) + \widetilde{\mathcal{T}} (f) (\rho (x) u) + \lambda f (\rho(x) u)  \bigg) + \bigg(  f \big( \rho (\mathfrak{T} x) (\mathcal{T} u)   \big)  + \widetilde{\mathcal{T}} (f) \big(  \rho(x) (\mathcal{T} u) \big) + \lambda f \big( \rho (x) (\mathcal{T}u) \big)  \bigg)  \nonumber \\
&= \lambda \bigg(   f (\rho (\mathfrak{T}x)u) - \cancel{\lambda f (\rho(x) u)} - \cancel{f \big( \mathcal{T}(\rho(x)u)  \big)} + \cancel{\lambda f (\rho(x)u)}   \bigg)  \nonumber \\
&~~+ f \bigg(  \mathcal{T} \big( \rho (\mathfrak{T}x) u \big) + \cancel{ \mathcal{T} \big(  \rho(x) (\mathcal{T}u)  \big)} + \cancel{ \mathcal{T} (\rho(x) u)} - \cancel{\lambda \rho(x) (\mathcal{T}u)} - \cancel{\mathcal{T} \big(\rho(x) (\mathcal{T}u) \big)}  + \cancel{\lambda \rho(x)(\mathcal{T} u)}   \bigg)  \nonumber  \\
&= \lambda ~f  (\rho (\mathfrak{T}x) u) + f \big(  \mathcal{T} (\rho (\mathfrak{T}x)u) \big).
\end{align}
It follows from (\ref{end-1}) and (\ref{end-2}) that $(\mathrm{End}(\mathcal{V}), \widetilde{\mathcal{T}})$ is a representation of $(\mathfrak{g}, \mathfrak{T})$.
\end{proof}

In the following, we construct the semidirect product in the context of $\lambda$-weighted Rota-Baxter Lie algebras.

\begin{prop}
Let $(\mathfrak{g}, \mathfrak{T})$ be a $\lambda$-weighted Rota-Baxter Lie algebra and $(\mathcal{V}, \mathcal{T})$ be a representation of it. Then $(\mathfrak{g} \oplus \mathcal{V}, \mathfrak{T} \oplus \mathcal{T})$ is a $\lambda$-weighted Rota-Baxter Lie algebra, where the Lie bracket on $\mathfrak{g} \oplus \mathcal{V}$ is given by the semidirect product
\begin{align}\label{semi-lie}
[(x,u), (y, v)]_\ltimes := ([x,y], \rho (x) v - \rho (y) u ),~ \text{ for } x, y \in \mathfrak{g}, u, v \in V.
\end{align}
\end{prop}

\begin{proof}
We have
\begin{align*}
&[ (\mathfrak{T} \oplus \mathcal{T}) (x,u), (\mathfrak{T} \oplus \mathcal{T})(y, v) ]_\ltimes \\
&= \big( [\mathfrak{T}(x), \mathfrak{T}(y)],~ \rho (\mathfrak{T}x) (\mathcal{T}v) - \rho (\mathfrak{T}y)(\mathcal{T}u) \big) \\
&= \big(  \mathfrak{T} [\mathfrak{T}(x), y],~ \mathcal{T}(  \rho (\mathfrak{T}x) v - \rho (\mathfrak{T}y) u) \big)   + \big( \mathfrak{T}[x, \mathfrak{T}(y)],~ \mathcal{T} (\rho (x) (\mathcal{T}v) - \rho(y)(\mathcal{T}u))   \big) \\
&~~+ \lambda~ \big( \mathfrak{T}[x,y],~ \mathcal{T} (\rho (x) v - \rho (y) u)   \big) \\
&= (\mathfrak{T} \oplus \mathcal{T}) \bigg( [(\mathfrak{T} \oplus \mathcal{T})(x,u), (y,v)]_\ltimes + [(x,u), (\mathfrak{T} \oplus \mathcal{T})(y,v)]_\ltimes + \lambda [(x,u), (y,v)]_\ltimes    \bigg).
\end{align*}
This shows that $\mathfrak{T} \oplus \mathcal{T}$ is a $\lambda$-weighted Rota-Baxter operator on the semidirect product Lie algebra. Hence the result follows.
\end{proof}

\begin{remark}
The converse of the above proposition is also true. More precisely, let $(\mathfrak{g}, \mathfrak{T})$ be a $\lambda$-weighted Rota-Baxter Lie algebra. Let $\mathcal{V}$ be a vector space and $\rho : \mathfrak{g} \rightarrow \mathrm{End}(\mathcal{V})$, $\mathcal{T} : \mathcal{V} \rightarrow \mathcal{V}$ be two linear maps. Then $(\mathcal{V}= (\mathcal{V}, \rho), \mathcal{T})$ is a representation of the $\lambda$-weighted Rota-Baxter Lie algebra $(\mathfrak{g}, \mathfrak{T})$ if and only if $(\mathfrak{g} \oplus \mathcal{V}, \mathfrak{T} \oplus \mathcal{T})$ is a $\lambda$-weighted Rota-Baxter Lie algebra, where $\mathfrak{g} \oplus \mathcal{V}$ is equipped with the bracket (\ref{semi-lie}).
\end{remark}

\begin{remark}
Let $\mathfrak{g}$ be a Lie algebra and $\mathcal{V}$ be a representation of it. Let $\mathfrak{T}: \mathfrak{g} \rightarrow \mathfrak{g}$ be a $\lambda$-weighted Rota-Baxter operator on $\mathfrak{g}$ which makes $(\mathfrak{g}, \mathfrak{T})$ a  $\lambda$-weighted Rota-Baxter Lie algebra. Intuitively, it follows from the above proposition that $\mathcal{T}: \mathcal{V} \rightarrow \mathcal{V}$ can be considered as a representation of $\mathfrak{T}$ with respect to the $\mathfrak{g}$-representation $\mathcal{V}$. In Section \ref{sec-rbp}, we call the tuple $(\mathfrak{T}, \mathcal{T})$ as $\lambda$-weighted Rota-Baxter paired operators. 
\end{remark}


\begin{prop}\label{defor-lie}
Let $(\mathfrak{g}, \mathfrak{T})$ be a $\lambda$-weighted Rota-Baxter Lie algebra. Then we have the followings.
\begin{itemize}
\item[(i)] The pair $(\mathfrak{g}, [~,~]_\mathfrak{T})$ is a Lie algebra, where
\begin{align*}
[x,y]_\mathfrak{T} := [\mathfrak{T}(x), y] + [x, \mathfrak{T}(y)] + \lambda [x, y], ~ \text{ for } x, y \in \mathfrak{g}.
\end{align*}
We denote this Lie algebra by $\mathfrak{g}_\mathfrak{T}$.
\item[(ii)] The pair $( \mathfrak{g}_\mathfrak{T}, \mathfrak{T} )$ is a $\lambda$-weighted Rota-Baxter Lie algebra and the map $\mathfrak{T} : \mathfrak{g}_\mathfrak{T} \rightarrow \mathfrak{g}$ is a morphism of $\lambda$-weighted Rota-Baxter Lie algebras.
\end{itemize}
\end{prop}

\begin{proof}
(i) This is a standard result that $\mathfrak{g}_\mathfrak{T} = (\mathfrak{g}, [~,~]_\mathfrak{T})$ is a Lie algebra. See for instance  \cite[Proposition 5.5]{das-weighted}. We also observe that
\begin{align*}
[\mathfrak{T}(x), \mathfrak{T}(y)]_\mathfrak{T} =~& [\mathfrak{T}^2(x), \mathfrak{T}(y)] + [\mathfrak{T}(x), \mathfrak{T}^2(y)] + \lambda [\mathfrak{T}(x), \mathfrak{T}(y)] \\
=~& \mathfrak{T} \big(  [\mathfrak{T}(x), y]_\mathfrak{T} + [x, \mathfrak{T}(y)]_\mathfrak{T} + \lambda [x,y]_\mathfrak{T}  \big)
\end{align*}
which shows that $\mathfrak{T}$ is a $\lambda$-weighted Rota-Baxter operator on the Lie algebra $\mathfrak{g}_\mathfrak{T}$.

(ii) Since $\mathfrak{T}$ is a $\lambda$-weighted Rota-Baxter operator on $\mathfrak{g}$, it follows from (\ref{lam-rota-id}) that
\begin{align*}
\mathfrak{T}([x,y]_\mathfrak{T}) = [\mathfrak{T}(x), \mathfrak{T}(y)], ~\text{ for } x, y \in \mathfrak{g}_\mathfrak{T}.
\end{align*}
This imples that $\mathfrak{T} : \mathfrak{g}_\mathfrak{T} \rightarrow \mathfrak{g}$ is a morphism of $\lambda$-weighted Rota-Baxter Lie algebras from $(\mathfrak{g}_\mathfrak{T}, \mathfrak{T})$ to $(\mathfrak{g}, \mathfrak{T})$.
\end{proof}

\begin{thm}
Let $(\mathfrak{g}, \mathfrak{T})$ be a $\lambda$-weighted Rota-Baxter Lie algebra and $(\mathcal{V}, \mathcal{T})$ be a representation of it. Define a map $\overline{\rho} : \mathfrak{g} \rightarrow \mathrm{End}(\mathcal{V})$ by
\begin{align*}
\overline{\rho}(x) u := \rho (\mathfrak{T}x) u + \rho (x) (\mathcal{T}u) + \lambda \rho (x) u, ~ \text{ for } x \in \mathfrak{g}, u \in \mathcal{V}.
\end{align*}
Then
\begin{itemize}
\item[(i)] $\overline{\rho}$ satisfies $\mathcal{T} ( \overline{\rho}(x) u) = \rho (\mathfrak{T}(x)) \mathcal{T}(u),$
\item[(ii)] $(\overline{\mathcal{V}} = (\mathcal{V}, \overline{\rho}), \mathcal{T})$ is a representation of the $\lambda$-weighted Rota-Baxter Lie algebra $(\mathfrak{g}_\mathfrak{T}, \mathfrak{T})$.
\end{itemize}
\end{thm}

\begin{proof}
The part (i) follows from (\ref{lam-rep}). To prove the part (ii) we first observe that 
\begin{align*}
&\overline{\rho}(x) \overline{\rho} (y) u - \overline{\rho}(y) \overline{\rho} (x) u \\
&= \rho (\mathfrak{T}x) \overline{\rho}(y) u + \rho (x) \mathcal{T}(   \overline{\rho}(y) u ) + \lambda \rho (x) \overline{\rho}(y) u
- \rho (\mathfrak{T}y) \overline{\rho}(x) u + \rho (y) \mathcal{T}(   \overline{\rho}(x) u ) + \lambda \rho (y) \overline{\rho}(x) u \\
&= \rho (\mathfrak{T}x) \rho (\mathfrak{T}y) u + \rho (\mathfrak{T}x) \rho (y) \mathcal{T}(u) + \lambda \rho (\mathfrak{T}x) \rho(y) u + \rho (x) \rho (\mathfrak{T}y) \mathcal{T}(u)  \\
& ~~ + \lambda \rho (x) \rho (\mathfrak{T}y) u + \lambda \rho (x) \rho (y) \mathcal{T}(u) + \lambda^2 \rho(x) \rho (y) u  \\
& ~~ - \rho (\mathfrak{T}y) \rho (\mathfrak{T}x) u - \rho (\mathfrak{T}y) \rho (x) \mathcal{T}(u) - \lambda \rho (\mathfrak{T}y) \rho(x) u - \rho (y) \rho (\mathfrak{T}x) \mathcal{T}(u) \\
& ~~ - \lambda \rho (y) \rho (\mathfrak{T}x) u - \lambda \rho (y) \rho (x) \mathcal{T}(u) - \lambda^2 \rho(y) \rho (x) u \\
&= \rho ([\mathfrak{T}x, \mathfrak{T}y])u + \rho \big( [\mathfrak{T}x, y] + [x, \mathfrak{T}y] + \lambda [x,y]  \big) \mathcal{T}(u) + \lambda \rho \big( [\mathfrak{T}x, y] + [x, \mathfrak{T}y] + \lambda [x,y]   \big) u \\
&= \rho (\mathfrak{T}[x,y]_\mathfrak{T}) u + \rho ([x,y]_\mathfrak{T}) \mathcal{T}(u) + \lambda \rho ([x,y]_\mathfrak{T}) u \\
&= \overline{\rho} ([x,y]_\mathfrak{T}) u.
\end{align*}
This shows that $\overline{\mathcal{V}} = (\mathcal{V}, \overline{\rho})$ is a representation of the Lie algebra $\mathfrak{g}_\mathfrak{T}$. Moreover, we have
\begin{align*}
\overline{\rho} (\mathfrak{T}x) \mathcal{T}(u) =~& \rho (\mathfrak{T}^2 (x)) \mathcal{T}(u) + \rho (\mathfrak{T}x) \mathcal{T}^2(u) + \lambda \rho (\mathfrak{T}x) \mathcal{T}(u) \\
=~& \mathcal{T} \big(  \overline{\rho}(\mathfrak{T}x) u + \overline{\rho}(x) \mathcal{T}(u) + \lambda \overline{\rho}(x) u   \big) ~~ \quad (\text{from } (i)).
\end{align*}
This shows that $(\overline{\mathcal{V}}, \mathcal{T})$ is a representation of $(\mathfrak{g}_\mathfrak{T}, \mathfrak{T})$.
\end{proof}

\begin{remark}
When considering the adjoint representation $(\mathfrak{g}, \mathfrak{T})$ of the $\lambda$-weighted Rota-Baxter Lie algebra $(\mathfrak{g}, \mathfrak{T})$, the representation of the above theorem  is the adjoint representation of the $\lambda$-weighted Rota-Baxter Lie algebra $(\mathfrak{g}_\mathfrak{T}, \mathfrak{T})$.
\end{remark}

In the following, we will prove another relevant result that will be useful in the next section to construct the cohomology of $\lambda$-weighted Rota-Baxter Lie algebras.

\begin{thm}\label{thm-coho}
Let $(\mathfrak{g}, \mathfrak{T})$ be a $\lambda$-weighted Rota-Baxter Lie algebra and $(\mathcal{V}, \mathcal{T})$ be a representation of it. Define a map $\widetilde{\rho} : \mathfrak{g} \rightarrow \mathrm{End}(\mathcal{V})$ by
\begin{align*}
\widetilde{\rho}(x) u = \rho (\mathfrak{T}(x)) u - \mathcal{T} (\rho (x) u), ~ \text{ for } x \in \mathfrak{g}, u \in \mathcal{V}.
\end{align*}
Then $\widetilde{\rho}$ defines a representation of the Lie algebra $\mathfrak{g}_\mathfrak{T}$ on $\mathcal{V}$. Moreover, $(\widetilde{\mathcal{V}} = (\mathcal{V}, \widetilde{\rho}), \mathcal{T})$ is a representation of the $\lambda$-weighted Rota-Baxter Lie algebra $(\mathfrak{g}_\mathfrak{T}, \mathfrak{T})$.
\end{thm}

\begin{proof}
For any $x, y \in \mathfrak{g}_\mathfrak{T}$ and $u \in \mathcal{V}$, we have
\begin{align*}
&\widetilde{\rho} (x) \widetilde{\rho}(y) u ~-~ \widetilde{\rho} (y) \widetilde{\rho}(x) u \\
&= \widetilde{\rho}(x) \big(  \rho (\mathfrak{T}y) u - \mathcal{T} (\rho (y) u) \big) ~-~ \widetilde{\rho}(y) \big(  \rho (\mathfrak{T}x) u - \mathcal{T} (\rho (x) u) \big) \\
&= \rho (\mathfrak{T}x) \rho (\mathfrak{T}y) u - \rho (\mathfrak{T}x) \mathcal{T}(\rho(y) u) - \mathcal{T} \big( \rho(x) \rho(\mathfrak{T}y)u - \rho(x) \mathcal{T}(\rho(y) u)  \big) \\
& ~~ - \rho (\mathfrak{T}y) \rho (\mathfrak{T}x) u + \rho (\mathfrak{T}y) \mathcal{T}(\rho(x) u) + \mathcal{T} \big( \rho(y) \rho(\mathfrak{T}x)u - \rho(y) \mathcal{T}(\rho(x) u)  \big) \\
&= \rho (\mathfrak{T}x) \rho (\mathfrak{T}y) u - \mathcal{T} \bigg(  \rho (\mathfrak{T}x) \rho(y) u + \cancel{\rho(x) \mathcal{T}(\rho(y) u)} + \lambda \rho(x) \rho(y) u \bigg) - \mathcal{T} \bigg( \rho(x) \rho(\mathfrak{T}y) u - \cancel{\rho(x) \mathcal{T}(\rho(y) u)}   \bigg) \\
& ~~ - \rho (\mathfrak{T}y) \rho (\mathfrak{T}x) u + \mathcal{T} \bigg(  \rho (\mathfrak{T}y) \rho(x) u + \cancel{\rho(y) \mathcal{T}(\rho(x) u)} + \lambda \rho(y) \rho(x) u \bigg) + \mathcal{T} \bigg( \rho(y) \rho(\mathfrak{T}x) u - \cancel{\rho(y) \mathcal{T}(\rho(x) u)}   \bigg) \\
&= \rho ([\mathfrak{T}x, \mathfrak{T}y])u ~-~ \mathcal{T} \big( \rho([\mathfrak{T}x, y] + [x, \mathfrak{T}y] + \lambda [x, y]) u  \big)\\
&= \rho (\mathfrak{T} [x,y]_\mathfrak{T})u ~-~ \mathcal{T} \big( \rho([x, y]_\mathfrak{T}) u  \big) \\
&= \widetilde{\rho} ([x,y]_\mathfrak{T})u.
\end{align*}
This shows that $\widetilde{\mathcal{V}} = (\mathcal{V}, \widetilde{\rho})$ defines a representation of the Lie algebra $\mathfrak{g}_\mathfrak{T}$. Moreover, we have
\begin{align*}
&\widetilde{\rho} (\mathfrak{T}x)( \mathcal{T} u) \\
&= \rho (\mathfrak{T}^2 (x)) (\mathcal{T}u) - \mathcal{T} (\rho (\mathfrak{T}x)(\mathcal{T}u)) \\
&= \mathcal{T} \big(  \rho (\mathfrak{T}^(x)) u + \rho (\mathfrak{T}x) (\mathcal{T}u)  + \lambda \rho (\mathfrak{T}x) u  \big) ~-~ \mathcal{T}^2 \big(     \rho (\mathfrak{T}x) u + \rho (x) (\mathcal{T}u) + \lambda \rho(x) u \big) \\
&= \mathcal{T} \big(  \rho (\mathfrak{T}^2(x)) u - \mathcal{T}(\rho (\mathfrak{T}x)u)  \big) + \mathcal{T} \big( \rho (\mathfrak{T}x) (\mathcal{T}u) - \mathcal{T} (\rho(x) (\mathcal{T}u))  \big) + \mathcal{T} \big(  \lambda \rho (\mathfrak{T}x) u - \lambda \mathcal{T}(\rho(x) u)  \big) \\
&= \mathcal{T} \big( \widetilde{\rho}(\mathfrak{T}x) u + \widetilde{\rho}(x)(\mathcal{T}u) + \lambda \widetilde{\rho} (x) u       \big)
\end{align*}
which shows that $(\widetilde{\mathcal{V}}, \mathcal{T})$ is a representation of $(\mathfrak{g}_\mathfrak{T}, \mathfrak{T})$.
\end{proof}

\section{Cohomology of weighted Rota-Baxter Lie algebras}\label{sec-3}
In this section, we first recall the Chevalley-Eilenberg cohomology of a Lie algebra with coefficients in a representation. Then we define the cohomology of a weighted Rota-Baxter Lie algebra with coefficients in a representation. This cohomology is obtained as a byproduct of the Chevalley-Eilenberg cohomology of the underlying Lie algebra with the cohomology of the weighted Rota-Baxter operator. Applications of cohomology to abelian extensions and formal $1$-parameter deformations are given in the next section. 

\medskip

Let $\mathfrak{g}$ be a Lie algebra and $\mathcal{V} = (\mathcal{V}, \rho)$ be a representation of it. The Chevalley-Eilenberg cohomology of $\mathfrak{g}$ with coefficients in $\mathcal{V}$ is given by the cohomology of the cochain complex $\{ C^*_\mathsf{CE} (\mathfrak{g}, \mathcal{V}), \delta_\mathsf{CE} \}$, where $C^n_\mathsf{CE} (\mathfrak{g}, \mathcal{V}) = \mathrm{Hom}(\wedge^n \mathfrak{g}, \mathcal{V})$ for $n \geq 0$ and the coboundary map $\delta_\mathsf{CE} : C^n_\mathsf{CE} (\mathfrak{g}, \mathcal{V}) \rightarrow C^{n+1}_\mathsf{CE} (\mathfrak{g}, \mathcal{V})$ given by
\begin{align*}
(\delta_\mathsf{CE} f) (x_1, \ldots, x_{n+1}) =~& \sum_{i=1}^{n+1} (-1)^{i+n} ~ \rho (x_i ) f (x_1, \ldots, \widehat{x_i}, \ldots, x_{n+1}) \\
&+ \sum_{1 \leq i < j \leq n+1} (-1)^{i+j+n+1} ~ f ([x_i, x_j], x_1, \ldots, \widehat{x_i}, \ldots, \widehat{x_j}, \ldots, x_{n+1}),
\end{align*}
for $f \in C^n_\mathsf{CE}(\mathfrak{g}, \mathcal{V})$ and $x_1, \ldots, x_{n+1} \in \mathfrak{g}$.

Let $(\mathfrak{g}, \mathfrak{T})$ be a $\lambda$-weighted Rota-Baxter Lie algebra and $(\mathcal{V}, \mathcal{T})$ be a representation of it. Then we have seen in Theorem \ref{thm-coho} that $\widetilde{\mathcal{V}} = (\mathcal{V}, \widetilde{\rho})$ is a representation of the Lie algebra $\mathfrak{g}_\mathfrak{T}$. Therefore, one can define the corresponding Chevalley-Eilenberg cohomology. More precisely, for each $n \geq 0$, we define
\begin{align*}
C^n_\mathsf{CE} (\mathfrak{g}_\mathfrak{T}, \widetilde{\mathcal{V}}) = \mathrm{Hom} (\wedge^n \mathfrak{g}, \mathcal{V})
\end{align*}
and a coboundary map $\partial_\mathsf{CE} : C^n_\mathsf{CE} (\mathfrak{g}_\mathfrak{T}, \widetilde{\mathcal{V}}) \rightarrow C^{n+1}_\mathsf{CE} (\mathfrak{g}_\mathfrak{T}, \widetilde{\mathcal{V}})$ given by
\begin{align*}
&(\partial_\mathsf{CE} f) (x_1, \ldots, x_{n+1}) \\
&= \sum_{i=1}^{n+1} (-1)^{i+n} ~ \widetilde{\rho} (x_i) f (x_1, \ldots, \widehat{x_i}, \ldots, x_{n+1}) \\
& ~~ + \sum_{1 \leq i < j \leq n+1} (-1)^{i+j+n+1} ~f ([x_i, x_j]_\mathfrak{T}, x_1, \ldots, \widehat{x_i}, \ldots, \widehat{x_j}, \ldots, x_{n+1})\\
&= \sum_{i=1}^{n+1} (-1)^{i+n} ~ \rho (\mathfrak{T}(x_i))   f (x_1, \ldots, \widehat{x_i}, \ldots, x_{n+1}) - \sum_{i=1}^{n+1} (-1)^{i+n} ~ \mathcal{T} \big( \rho (x_i)  f (x_1, \ldots, \widehat{x_i}, \ldots, x_{n+1}) \big)   \\
& ~~ + \sum_{1 \leq i < j \leq n+1} (-1)^{i+j+n+1} ~f ([\mathfrak{T}(x_i), x_j] + [x_i, \mathfrak{T}(x_j)] + \lambda [x_i, x_j], x_1, \ldots, \widehat{x_i}, \ldots, \widehat{x_j}, \ldots, x_{n+1}).
\end{align*}
Then $\{ C^*_\mathsf{CE} (\mathfrak{g}_\mathfrak{T}, \widetilde{\mathcal{V}}), \partial_\mathsf{CE} \}$ is a cochain complex. The corresponding cohomology groups are called the cohomology of $\mathfrak{T}$ with coefficients in the representation $\mathcal{T}$.

\begin{remark}
When $(\mathcal{V}, \mathcal{T}) = (\mathfrak{g}, \mathfrak{T})$ is the adjoint representation of the $\lambda$-weighted Rota-Baxter Lie algebra $(\mathfrak{g}, \mathfrak{T})$, one may consider the cohomology of $\mathfrak{T}$ with coefficients in the representation $\mathfrak{T}$ itself. In \cite{das-weighted} the author defines the cohomology of a $\lambda$-weighted Rota-Baxter operators motivated from their Maurer-Cartan characterizations. It follows that the cohomology of $\mathfrak{T}$ in the sense of \cite{das-weighted} is isomorphic to our cohomology of $\mathfrak{T}$ with coefficients in the representation $\mathfrak{T}$ itself.
\end{remark}

\medskip

We will now in a position to define the cohomology of the $\lambda$-weighted Rota-Baxter Lie algebra $(\mathfrak{g}, \mathfrak{T})$ with coefficients in the representation $(\mathcal{V}, \mathcal{T})$. We first consider the two cochain complexes, namely,

\medskip
\begin{itemize}
\item[$\diamondsuit$] the Chevalley-Eilenberg cochain complex $\{ C^*_\mathsf{CE} (\mathfrak{g}, V), \delta_{\mathsf{CE}} \}$ defining the cohomology of the Lie algebra with coefficients in the representation $\mathcal{V}$,

\medskip

\item[$\diamondsuit$] the complex $\{ C^*_\mathsf{CE} (\mathfrak{g}_\mathfrak{T}, \widetilde{V}), \partial_\mathsf{CE} \}$ defining the cohomology of $\mathfrak{T}$ with coefficients in the representation $\mathcal{T}.$
\end{itemize}

\medskip

The following result is similar to \cite[Proposition 5.1]{zhou}.

\begin{prop}\label{cochain-map}
The collection of maps $\{ \Phi^n : C^n_\mathsf{CE} (\mathfrak{g}, V) \rightarrow C^n_\mathsf{CE} (\mathfrak{g}_\mathfrak{T}, \widetilde{V}) \}_{n \geq 0}$ defined by
\begin{align*}
&\Phi^0 = \mathrm{id}_\mathcal{V}, ~~ \text{ and } \\
&\Phi^n (f) (x_1, \ldots, x_n) = f (\mathfrak{T}(x_1), \ldots, \mathfrak{T}(x_n)) - \sum_{k=0}^{n-1} \lambda^{n-k-1}  \sum_{i_1 < \cdots < i_k} \mathcal{T} \circ f \big( x_1, \ldots, \mathfrak{T}(x_{i_1}), \ldots, \mathfrak{T}(x_{i_k}), \ldots, x_n    \big)
\end{align*}
is a morphism of cochain complexes from $\{ C^*_\mathsf{CE} (\mathfrak{g}, V), \delta_{\mathsf{CE}} \}$ to $\{ C^*_\mathsf{CE} (\mathfrak{g}_\mathfrak{T}, \widetilde{V}), \partial_\mathsf{CE} \}$, i.e., $ \partial_\mathsf{CE} \circ \Phi^n = \Phi^{n+1} \circ \delta_\mathsf{CE}$, for $n \geq 0$.
\end{prop}

\medskip

Let $(\mathfrak{g}, \mathfrak{T})$ be a $\lambda$-weighted Rota-Baxter Lie algebra and $(\mathcal{V}, \mathcal{T})$ be a representation of it. For each $n \geq 0$, we define an abelian group $C^n_\mathsf{RB} (\mathfrak{g}, V)$ by
\begin{align*}
C^n_\mathsf{RB} (\mathfrak{g}, \mathcal{V}) = \begin{cases} C^0_\mathsf{CE} (\mathfrak{g}, \mathcal{V}) = \mathcal{V}    & \text{ if } n = 0 \\
C^n_\mathsf{CE}(\mathfrak{g}, \mathcal{V}) \oplus C^{n-1}_\mathsf{CE} (\mathfrak{g}_\mathfrak{T}, \widetilde{V}) = \mathrm{Hom} (\wedge^n \mathfrak{g}, \mathcal{V}) \oplus \mathrm{Hom} (\wedge^{n-1} \mathfrak{g}, \mathcal{V}) & \text{ if } n \geq 1
\end{cases}
\end{align*}
and a map $\delta_\mathsf{RB} : C^n_\mathsf{RB} (\mathfrak{g}, \mathcal{V}) \rightarrow C^{n+1}_\mathsf{RB} (\mathfrak{g}, \mathcal{V})$ by
\begin{align*}
\delta_\mathsf{RB} (v) =~& ( \delta_{\mathsf{CE}} (v), -v), ~\text{ for } v \in C^n_\mathsf{RB} (\mathfrak{g}, \mathcal{V}) = \mathcal{V},\\
\delta_\mathsf{RB} (f,g) =~& (\delta_\mathsf{CE} (f), ~-\partial_\mathsf{CE} (g) - \Phi^n (f) ), ~ \text{ for } (f,g) \in C^n_\mathsf{RB} (\mathfrak{g}, \mathcal{V}).
\end{align*}
Note that
\begin{align*}
(\delta_\mathsf{RB})^2 (v) =~& \delta_\mathsf{RB} (\delta_\mathsf{CE} (v), -v ) \\
=~& \big( (\delta_\mathsf{CE})^2 (v), \partial_\mathsf{CE} (v) - \Phi^2 \circ \delta_\mathsf{CE} (v) \big) = 0 \quad (\because ~ \partial_\mathsf{CE} \circ \Phi^0 = \Phi^1 \circ \delta_\mathsf{CE})
\end{align*}
and
\begin{align*}
(\delta_\mathsf{RB})^2 (f,g) =~& \delta_\mathsf{RB} \big( \delta_\mathsf{CE} (f), ~ - \partial_\mathsf{CE} (g) - \Phi^n (f)  \big) \\
=~& \big(  (\delta_\mathsf{CE})^2 (f) , ~ (\partial_\mathsf{CE})^2 (g) + \partial_\mathsf{CE} \circ \Phi^n (f)  - \Phi^{n+1} \circ \delta_\mathsf{CE} (f)  \big)
= 0 \quad  (\because ~ \partial_\mathsf{CE} \circ \Phi^n = \Phi^{n+1} \circ \delta_\mathsf{CE}).
\end{align*}
This shows that $\{ C^*_\mathsf{RB} (\mathfrak{g}, \mathcal{V}), \delta_\mathsf{RB} \}$ is a cochain complex. Let $Z^n_\mathsf{RB} (\mathfrak{g}, \mathcal{V})$ and $B^n_\mathsf{RB} (\mathfrak{g}, \mathcal{V})$ denote the space of $n$-cocycles and $n$-coboundaries, respectively. Then we have $B^n_\mathsf{RB} (\mathfrak{g}, \mathcal{V}) \subset Z^n_\mathsf{RB} (\mathfrak{g}, \mathcal{V})$, for $n \geq 0$. The corresponding quotients
\begin{align*}
H^n_\mathsf{RB} (\mathfrak{g}, \mathcal{V}) := \frac{  Z^n_\mathsf{RB} (\mathfrak{g}, \mathcal{V})}{   B^n_\mathsf{RB} (\mathfrak{g}, \mathcal{V})}, ~\text{ for } n \geq 0
\end{align*}
are called the cohomology of the $\lambda$-weighted Rota-Baxter Lie algebra $(\mathfrak{g}, \mathfrak{T})$ with coefficients in the representation $(\mathcal{V}, \mathcal{T})$.

\medskip

Observe that there is a short exact sequence of cochain complexes
\begin{align*}
0 \rightarrow C^{* -1}_\mathsf{CE} (\mathfrak{g}_\mathfrak{T}, \widetilde{\mathcal{V}}) \xrightarrow{i} C^*_\mathsf{RB} (\mathfrak{g}, \mathcal{V}) \xrightarrow{p} C^*_\mathsf{CE} (\mathfrak{g}, \mathcal{V}) \rightarrow 0
\end{align*}
given by $i (g) = (0, (-1)^{n-1} g)$ and $p(f,g) = f$, for $(f, g) \in C^n_\mathsf{RB}(\mathfrak{g}, \mathcal{V})$. This short exact sequence induces the following long exact sequence on cohomology groups
\begin{align*}
0 \rightarrow H^0_\mathsf{RB}(\mathfrak{g}, \mathcal{V}) \rightarrow H^0_\mathsf{CE} (\mathfrak{g}, \mathcal{V}) \rightarrow H^0_\mathsf{CE} (\mathfrak{g}_\mathfrak{T}, \widetilde{\mathcal{V}}) 
\rightarrow H^1_\mathsf{RB}(\mathfrak{g}, \mathcal{V}) \rightarrow H^1_\mathsf{CE} (\mathfrak{g}, \mathcal{V}) \rightarrow H^1_\mathsf{CE} (\mathfrak{g}_\mathfrak{T}, \widetilde{\mathcal{V}})  \rightarrow \cdots .
\end{align*}

\begin{prop}
Let $(\mathfrak{g}, \mathfrak{T})$ be a $\lambda$-weighted Rota-Baxter Lie algebra and $(\mathcal{V}, \mathcal{T})$ be a representation of it. Then the cohomology of $(\mathfrak{g}, \mathfrak{T})$ with coefficients in the representation $(\mathcal{V}, \mathcal{T})$ is isomorphic to the cohomology of $(\mathfrak{g}, - \lambda \mathrm{id}_\mathfrak{g} - \mathfrak{T})$ with coefficients in the representation $(\mathcal{V}, - \lambda \mathrm{id}_\mathcal{V} - \mathcal{T})$.
\end{prop}

\begin{proof}
Let $\mathfrak{T}' = - \lambda \mathrm{id}_\mathfrak{g} - \mathfrak{T}$ and $\mathcal{T}' = - \lambda \mathrm{id}_\mathcal{V} - \mathcal{T}$. Then it is easy to see that the Lie algebra $\mathfrak{g}_{\mathfrak{T}'}$ of Proposition \ref{defor-lie} (i) is negative to the Lie algebra $\mathfrak{g}_\mathfrak{T}$. Moreover, the representation $\widetilde{\rho}'$ of the Lie algebra $\mathfrak{g}_{\mathfrak{T}'}$  on $\mathcal{V}$ (we denote this representation by $\widetilde{\rho}'$) as of Theorem \ref{thm-coho} is given by the negative of the representation $\widetilde{\rho}$ of the Lie algebra $\mathfrak{g}_\mathfrak{T}$ on $\mathcal{V}$. Hence the corresponding Chevalley-Eilenberg differentials
\begin{align*}
\partial_\mathsf{CE} : C^n_\mathsf{CE} (\mathfrak{g}_\mathfrak{T}, \widetilde{\mathcal{V}}) \rightarrow  C^{n+1}_\mathsf{CE} (\mathfrak{g}_\mathfrak{T}, \widetilde{\mathcal{V}})  ~~~~~ \text{ and } ~~~~~ \partial_\mathsf{CE}' : C^n_\mathsf{CE} (\mathfrak{g}_{\mathfrak{T}'}, \widetilde{\mathcal{V}}') \rightarrow  C^{n+1}_\mathsf{CE} (\mathfrak{g}_{\mathfrak{T}'}, \widetilde{\mathcal{V}}')
\end{align*}
are related by $\partial_\mathsf{CE}' f = - \partial_\mathsf{CE} f$. Finally, if $(\Phi')^n : C^n_\mathsf{CE} (\mathfrak{g}, \mathcal{V}) \rightarrow C^n_\mathsf{CE} (\mathfrak{g}_{\mathfrak{T}'}, \widetilde{\mathcal{V}}')$ be the map as of Proposition \ref{cochain-map} replacing $\mathfrak{T}$ by $\mathfrak{T}'$ and replacing $\mathcal{T}$ by $\mathcal{T}'$, then we have $(\Phi')^n (f) = (-1)^n ~ \Phi^n (f)$. 

For each $n \geq 0$, define an isomorphism of vector spaces
\begin{align*}
\Xi^n : C^n_\mathsf{CE} (\mathfrak{g}, \mathcal{V}) \oplus C^{n-1}_\mathsf{CE} (\mathfrak{g}_\mathfrak{T}, \widetilde{\mathcal{V}}) ~\rightarrow ~ C^n_\mathsf{CE} (\mathfrak{g}, \mathcal{V}) \oplus C^{n-1}_\mathsf{CE} (\mathfrak{g}_{\mathfrak{T}'}, \widetilde{\mathcal{V}}') ~~~ \text{ by }  ~~~ \Xi^n (f,g) = (f, (-1)^{n-1} g).
\end{align*}
Moreover,
\begin{align*}
(\delta_\mathsf{RB}' \circ \Xi^n) (f,g) =  \delta_\mathsf{RB}' (f, (-1)^{n-1} g) =~& \big( \delta_\mathsf{CE} f, ~(-1)^n \partial_\mathsf{CE}' g  - (\Phi')^n (f) \big) \\
=~& \big(   \delta_\mathsf{CE} f,~ (-1)^{n+1} \partial_\mathsf{CE}g - (-1)^n \Phi^n (f) \big) \\
=~& (\Xi^{n+1} \circ \delta_\mathsf{RB}) (f,g).
\end{align*}
This shows that the collection of maps $\{ \Xi^n \}_{n \geq 0}$ commute with the respective coboundary maps. Hence they induce isomorphism on cohomology.
\end{proof}

\subsection{$H^0$ and $H^1$}

Let $(\mathfrak{g}, \mathfrak{T})$ be a $\lambda$-weighted Rota-Baxter Lie algebra and $(\mathcal{V}, \mathcal{T})$ be a representation of it.
An element $v \in \mathcal{V}$ is in $Z^0_\mathsf{RB}(\mathfrak{g}, \mathcal{V})$ if and only if $(\delta_\mathsf{CE}(v) , -v) = 0$. This holds only when $v=0$. Therefore, it follows from the definition that $H^0_\mathsf{RB}(\mathfrak{g}, \mathcal{V}) = 0$.

\medskip

A pair $(\gamma, v) \in \mathrm{Hom}(\mathfrak{g}, \mathcal{V}) \oplus \mathcal{V}$ is said to be a derivation on the weighted Rota-Baxter Lie algebra $(\mathfrak{g}, \mathfrak{T})$ with coefficients in the representation $(\mathcal{V}, \mathcal{T})$ if they satisfies
\begin{align*}
\gamma ([x,y]) =~& \rho (x) (\gamma (y)) - \rho (y) (\gamma (x)),\\
\gamma (\mathfrak{T}(x)) - \mathcal{T} (\gamma (x)) =~& \mathcal{T} (\rho (x) v) - \rho (\mathfrak{T}(x)) v, ~ \text{ for } x, y \in \mathfrak{g}.
\end{align*}
It follows from the first condition that $\gamma$ is a derivation on the Lie algebra $\mathfrak{g}$ with coefficients in $\mathcal{V}$. The second condition says that the obstruction of vanishing $\gamma \circ \mathfrak{T} - \mathcal{T} \circ \gamma$ is measured by the presence of $v$. We denote the set of all derivations by $\mathrm{Der}(\mathfrak{g}, \mathcal{V})$.

A derivation is said to be inner if it is of the form $(-\delta_\mathsf{CE}(v) , v)$, for some $v \in \mathcal{V}$. The set of all inner derivations are denoted by $\mathrm{InnDer}(\mathfrak{g}, \mathcal{V}).$

It follows from the definition that $Z^1_\mathsf{RB}(\mathfrak{g}, \mathcal{V}) = \mathrm{Der}(\mathfrak{g}, \mathcal{V})$ and $B^1_\mathsf{RB}(\mathfrak{g}, \mathcal{V}) = \mathrm{InnDer}(\mathfrak{g}, \mathcal{V})$. Hence we have $H^1_\mathsf{RB}(\mathfrak{g}, \mathcal{V}) = \frac{ \mathrm{Der}(\mathfrak{g}, \mathcal{V})  }{ \mathrm{InnDer}(\mathfrak{g}, \mathcal{V}) }$, the space of outer derivations.

\subsection{Relation with the cohomology of weighted Rota-Baxter associative algebras}

In \cite{zhou} Wang and Zhou defined the cohomology of a weighted Rota-Baxter associative algebra with coefficients in a Rota-Baxter bimodule. In this subsection, we show that our cohomology is related to the cohomology of \cite{zhou} by suitable skew-symmetrization.

\begin{defn}
(i) A $\lambda$-weighted Rota-Baxter associative algebra is a pair $(\mathfrak{A}, \mathfrak{R})$ in which $\mathfrak{A}$ is an associative algebra and $\mathfrak{R} : \mathfrak{A} \rightarrow \mathfrak{A}$ is a linear map satisfying
\begin{align*}
\mathfrak{R}(a) \cdot \mathfrak{R}(b) = \mathfrak{R} \big(  \mathfrak{R}(a) \cdot b + a \cdot \mathfrak{R}(b) + \lambda a \cdot b    \big), ~ \text{ for } a, b \in \mathfrak{A}.
\end{align*}

(ii) Let $(\mathfrak{A}, \mathfrak{R})$ be a $\lambda$-weighted Rota-Baxter associative algebra.  A Rota-Baxter bimodule over it consists of  a pair $(\mathcal{M}, \mathcal{R})$ in which $\mathcal{M}$ is an $\mathfrak{A}$-bimodule (denote both left and right actions by cdot)  and $\mathcal{R} : \mathcal{M} \rightarrow \mathcal{M}$ is a linear map satisfying for $a \in \mathfrak{A}$, $m \in \mathcal{M}$,
\begin{align*}
\mathfrak{R}(a) \cdot \mathcal{R}(m) =~& \mathcal{R} \big( \mathfrak{R}(a) \cdot m + a \cdot \mathcal{R}(m) + \lambda ~a \cdot m   \big), \\
\mathcal{R}(m) \cdot \mathfrak{R}(a) =~& \mathcal{R} \big(   \mathcal{R}(m) \cdot a + m \cdot \mathfrak{R}(a) + \lambda ~ m \cdot a ).
\end{align*}
\end{defn}

\begin{remark}
It follows from the above definition that any $\lambda$-weighted Rota-Baxter associative algebra $(\mathfrak{A}, \mathfrak{R})$ is a Rota-Baxter bimodule over itself. This is called the adjoint Rota-Baxter bimodule.
\end{remark}

The following result is straightforward.

\begin{prop}
Let $(\mathfrak{A}, \mathfrak{R})$ be a $\lambda$-weighted Rota-Baxter associative algebra. Then $(\mathfrak{A}_c, \mathfrak{R})$ is a $\lambda$-weighted Rota-Baxter Lie algebra, where $\mathfrak{A}_c$ is the vector space $\mathfrak{A}$ with the commutator Lie bracket
\begin{align*}
[a,b]_c = a \cdot b - b \cdot a, ~ \text{ for } a, b \in \mathfrak{A}_c.
\end{align*}
(This is called the skew-symmetrization). Moreover, if $(\mathcal{M}, \mathcal{R})$ is a Rota-Baxter bimodule over $(\mathfrak{A}, \mathfrak{R})$, then $(\mathcal{M}_c, \mathcal{R})$ is a representation of the $\lambda$-weighted Rota-Baxter Lie algebra $(\mathfrak{A}_c, \mathfrak{R})$, where $\mathcal{M}_c = \mathcal{M}$ as a vector space and the representation of the Lie algebra $\mathfrak{A}_c$ on $\mathcal{M}_c$ is given by $\rho (a) (m) = a \cdot m - m \cdot a$, for $a \in \mathfrak{A}_c$ and $m \in \mathcal{M}_c$.
\end{prop}

\medskip

Let $(\mathfrak{A}, \mathfrak{R})$ be a $\lambda$-weighted Rota-Baxter associative algebra. Then it is known that $(\mathfrak{A}, *_\mathfrak{R})$ is an associative algebra \cite{fard}, where
\begin{align*}
a *_\mathfrak{R} b = \mathfrak{R}(a) \cdot b + a \cdot \mathfrak{R}(b) + \lambda ~ a \cdot b, ~ \text{ for } a, b \in \mathfrak{A}.
\end{align*}
We denote this associative algebra by $\mathfrak{A}_\mathfrak{R}$. Moreover, if $(\mathcal{M}, \mathcal{R})$ is a Rota-Baxter bimodule over the $\lambda$-weighted Rota-Baxter associative algebra $(\mathfrak{A}, \mathfrak{R})$, then it has been observed in \cite{zhou} that $\mathcal{M}$ carries a bimodule structure over the associative algebra $\mathfrak{A}_\mathfrak{R}$ with left and right actions
\begin{align*}
a ~\widetilde{\cdot }~ m =  \mathfrak{R}(a) \cdot m - \mathcal{R} (a \cdot m) ~~ \text{ and } ~~   m  ~\widetilde{\cdot }~  a = m \cdot \mathfrak{R}(a) - \mathcal{R} ( m \cdot a), ~ \text{ for } a \in \mathfrak{A}_\mathfrak{R} ~\text{ and } m \in \mathcal{M}.
\end{align*}
Denote this $\mathfrak{A}_\mathfrak{R}$-bimodule by $\widetilde{\mathcal{M}}$. To define the cohomology of the $\lambda$-weighted Rota-Baxter associative algebra $(\mathfrak{A}, \mathfrak{R})$ with coefficients in the Rota-Baxter bimodule $(\mathcal{M}, \mathcal{R})$, the authors in \cite{zhou} considered two Hochschild cochain complexes, namely $\{ C^*_\mathsf{H} (\mathfrak{A}, \mathcal{M}), \delta_\mathsf{H} \}$ and $\{ C^*_\mathsf{H} (\mathfrak{A}_\mathfrak{R}, \widetilde{M}), \partial_\mathsf{H} \}$. The first one is the Hochschild complex of the given algebra $\mathfrak{A}$ with coefficients in the bimodule $\mathcal{M}$, whereas the second one is the Hochschild complex of the algebra $\mathfrak{A}_\mathfrak{R}$ with coefficients in the bimodule $\widetilde{\mathcal{M}}$. They proved that the collection $\{ \Psi^n :   C^n_\mathsf{H} (\mathfrak{A}, \mathcal{M}) \rightarrow  C^n_\mathsf{H} (\mathfrak{A}_\mathfrak{R}, \widetilde{M}) \}_{n \geq 0}$ of maps
given by
\begin{align*}
&\Psi^0 = \mathrm{id}_\mathcal{M}, ~~ \text{ and } \\
&\Psi^n (f) (a_1, \ldots, a_n) = f (\mathfrak{R}(a_1), \ldots, \mathfrak{R}(a_n)) - \sum_{k=0}^{n-1} \lambda^{n-k-1}  \sum_{i_1 < \cdots < i_k} \mathcal{R} \circ f \big( a_1, \ldots, \mathfrak{R}(a_{i_1}), \ldots, \mathfrak{R}(a_{i_k}), \ldots, a_n    \big)
\end{align*}
defines a morphism of cochain complexes. This allows the authors to consider the cochain complex $\{ C^*_\mathsf{RB}(\mathfrak{A}, \mathcal{M}), \delta_{\mathsf{RB}}^{\mathsf{ass}} \}$, where
\begin{align*}
C^0_\mathsf{RB} (\mathfrak{A}, \mathcal{M}) = C^0_\mathsf{H} (\mathfrak{A}, \mathcal{M}) = \mathcal{M} ~~ \text{ and } ~~ C^n_\mathsf{RB} (\mathfrak{A}, \mathcal{M}) = C^n_\mathsf{H} (\mathfrak{A}, \mathcal{M}) \oplus C^{n-1}_\mathsf{H} (\mathfrak{A}_\mathfrak{R}, \widetilde{\mathcal{M}}), ~ \text{ for } n \geq 1
\end{align*}
and $\delta_\mathsf{RB}^{\mathsf{ass}} : C^n_\mathsf{RB} (\mathfrak{A}, \mathcal{M}) \rightarrow C^{n+1}_\mathsf{RB} (\mathfrak{A}, \mathcal{M})$ given by
\begin{align*}
\delta_\mathsf{RB}^{\mathsf{ass}} = (\delta_\mathsf{H}(f) , ~- \partial_\mathsf{H} (g) - \Psi^n (f)), ~ \text{ for } (f, g) \in C^n_\mathsf{RB}(\mathfrak{A}, \mathcal{M}).
\end{align*}
The corresponding cohomology groups are called the cohomology of the $\lambda$-weighted Rota-Baxter associative algebra $(\mathfrak{A}, \mathfrak{R})$ with coefficients in the Rota-Baxter bimodule $(\mathcal{M}, \mathcal{R})$.

\medskip

To find the connection between the cohomology of a weighted Rota-Baxter associative algebra and the cohomology of the corresponding skew-symmetrized weighted Rota-Baxter Lie algebra, we need the following result.

\begin{prop}
Let $(\mathfrak{A}, \mathfrak{R})$ be a $\lambda$-weighted Rota-Baxter associative algebra and $(\mathcal{M}, \mathcal{R})$ be a Rota-Baxter bimodule. Then $(\mathfrak{A}_c)_\mathfrak{R} = (\mathfrak{A}_\mathfrak{R})_c$ as a Lie algebra. Here $(\mathfrak{A}_c)_\mathfrak{R}$ is the Lie algebra structure on $\mathfrak{A}$ induced by the $\lambda$-weighted Rota-Baxter operator $\mathfrak{R}$ on the Lie algebra $\mathfrak{A}_c$ as in Proposition \ref{defor-lie} (i). Moreover, the representation of the Lie algebra $(\mathfrak{A}_c)_\mathfrak{R}$ on $\widetilde{\mathcal{M}_c}$ and the representation of the Lie algebra $(\mathfrak{A}_\mathfrak{R})_c$ on $(\widetilde{\mathcal{M}})_c$ coincide.
\end{prop}

It is known that the standard skew-symmetrization gives rise to a morphism from the Hochschild cochain complex of an associative algebra to the Chevalley-Eilenberg cochain complex of the corresponding skew-symmetrized Lie algebra \cite{loday-book}. Hence the following diagrams commute
\begin{align*}
\xymatrix{
C^n_\mathsf{H} (\mathfrak{A}, \mathcal{M}) \ar[d]_{S_n} \ar[r]^{\delta_\mathsf{H}} & C^{n+1}_\mathsf{H} (\mathfrak{A}, \mathcal{M}) \ar[d]^{S_{n+1}} & & C^n_\mathsf{H} (\mathfrak{A}_\mathfrak{R}, \widetilde{\mathcal{M}}) \ar[d]_{S_n} \ar[r]^{\partial\mathsf{H}} & C^{n+1}_\mathsf{H} (\mathfrak{A}_\mathfrak{R}, \widetilde{\mathcal{M}}) \ar[d]^{S_{n+1}}\\
C^n_\mathsf{CE}(\mathfrak{A}_c, \mathcal{M}_c) \ar[r]_{\delta_\mathsf{CE}} & C^{n+1}_\mathsf{CE} (\mathfrak{A}_c, \mathcal{M}_c) & & C^n_\mathsf{CE} ((\mathfrak{A}_\mathfrak{R})_c , \widetilde{M}_c ) \ar[r]_{\partial_\mathsf{CE}} & C^{n+1}_\mathsf{CE} ( (\mathfrak{A}_\mathfrak{R})_c , \widetilde{M}_c) .
}
\end{align*}
Here $S_*$ are the skew-symmetrization maps. As a consequence, we get the following.

\begin{thm}
Let $(\mathfrak{A}, \mathfrak{R})$  be a $\lambda$-weighted Rota-Baxter associative algebra and $(\mathcal{M}, \mathcal{R})$ be a Rota-Baxter bimodule. Then the collection of maps
\begin{align*}
\mathbbm{S}_n: C^n_\mathsf{RB}(\mathfrak{A}, \mathcal{M}) \rightarrow  C^n_\mathsf{RB}(\mathfrak{A}_c, \mathcal{M}_c), ~\text{ for } n \geq 0
\end{align*}
defined by $\mathbbm{S}_0 = \mathrm{id}_\mathcal{M}$ and $\mathbbm{S}_n = (S_n, S_{n-1})$ for $n \geq 1$, induces a morphism from the cohomology of $(\mathfrak{A}, \mathfrak{R})$ with coefficients in the Rota-Baxter bimodule $(\mathcal{M}, \mathcal{R})$ to the cohomology of $(\mathfrak{A}_c, \mathfrak{R})$ with coefficients in the representation $(\mathcal{M}_c, \mathcal{R}).$ 
\end{thm}

\begin{proof}
We only need to check that the maps $\{ \mathbbm{S}_n \}_{n \geq 0}$ commute with corresponding coboundary maps. For $(f, g) \in C^n_\mathsf{RB}(\mathfrak{A}, \mathcal{M})$,
\begin{align*}
(\delta_{\mathsf{RB}} \circ \mathbbm{S}_n) (f, g) =~& \delta_\mathsf{RB} (S_n f, S_{n-1} g)\\
=~& \big( \delta_\mathsf{CE} \circ S_n (f),~ - \partial_\mathsf{CE} \circ S_{n-1} (g) - (\Phi^n \circ S_n)(f)   \big)\\
=~& \big(  S_{n+1} \circ \delta_\mathsf{H} (f) , ~ - S_n \circ \partial_\mathsf{H}(g) - S_n \circ \Psi^n (f) \big)  \qquad (\because ~ \Phi^n \circ S_n = S_n \circ \Psi^n)\\
=~& \mathbbm{S}_{n+1} \big(  \delta_\mathsf{H} (f) , ~- \partial_\mathsf{H}(g) - \Psi^n (f)  \big)  = (\mathbbm{S}_{n+1} \circ \delta_\mathsf{RB}^\mathsf{ass})(f,g).
\end{align*}
Hence the result follows.
\end{proof}

\section{Applications of cohomology}\label{sec-4}

In this section, we study abelian extensions and formal $1$-parameter deformations of weighted Rota-Baxter Lie algebras in terms of cohomology.

\subsection{Abelian extensions and $H^2$}

Let $(\mathfrak{g}, \mathfrak{T})$ be a $\lambda$-weighted Rota-Baxter Lie algebra. Let $(\mathcal{V}, \mathcal{T})$ be a pair of a vector space $\mathcal{V}$ and a linear map $\mathcal{T} : \mathcal{V} \rightarrow \mathcal{V}$. Note that $(\mathcal{V}, \mathcal{T})$ can be considered as a $\lambda$-weighted Rota-Baxter Lie algebra where the Lie bracket on $\mathcal{V}$ is assumed to be trivial.

\begin{defn}
An abelian extension of  $(\mathfrak{g}, \mathfrak{T})$ by $(\mathcal{V}, \mathcal{T})$ is a short exact sequence of morphisms of $\lambda$-weighted Rota-Baxter Lie algebras 
\begin{align}\label{abelian}
\xymatrix{
0 \ar[r] & (\mathcal{V}, \mathcal{T}) \ar[r]^{i} & (\widehat{\mathfrak{g}}, \widehat{\mathfrak{T}}) \ar[r]^p & (\mathfrak{g}, \mathfrak{T}) \ar[r] & 0.
}
\end{align}
In this case, we say that $(\widehat{\mathfrak{g}}, \widehat{\mathfrak{T}})$ is an abelian extension of $({\mathfrak{g}}, {\mathfrak{T}})$ by $(\mathcal{V}, \mathcal{T}).$
\end{defn}

\begin{defn}
Let $(\widehat{\mathfrak{g}}, \widehat{\mathfrak{T}})$ and $(\widehat{\mathfrak{g}}', \widehat{\mathfrak{T}}')$ be two abelian extensions of  $({\mathfrak{g}}, {\mathfrak{T}})$ by $(\mathcal{V}, \mathcal{T}).$ They are said to be isomorphic if there exists an isomorphism $\phi : (\widehat{\mathfrak{g}}, \widehat{\mathfrak{T}}) \rightarrow (\widehat{\mathfrak{g}}', \widehat{\mathfrak{T}}')$ of $\lambda$-weighted Rota-Baxter Lie algebras which makes the following diagram commutative
\begin{align}\label{abelian-iso}
\xymatrix{
0 \ar[r] & (\mathcal{V}, \mathcal{T}) \ar@{=}[d] \ar[r]^{i} & (\widehat{\mathfrak{g}}, \widehat{\mathfrak{T}}) \ar[d]^\phi \ar[r]^p & (\mathfrak{g}, \mathfrak{T}) \ar@{=}[d] \ar[r] & 0 \\
0 \ar[r] & (\mathcal{V}, \mathcal{T}) \ar[r]_{i'} & (\widehat{\mathfrak{g}}', \widehat{\mathfrak{T}}') \ar[r]_{p'} & (\mathfrak{g}, \mathfrak{T}) \ar[r] & 0 
}
\end{align}
\end{defn}

Let $(\widehat{\mathfrak{g}}, \widehat{\mathfrak{T}})$ be an abelian extension of $({\mathfrak{g}}, {\mathfrak{T}})$ by $(\mathcal{V}, \mathcal{T})$ as of (\ref{abelian}). A section of the map $p$ is a linear map $s : \mathfrak{g} \rightarrow \widehat{ \mathfrak{g}}$ satisfying $p \circ s = \mathrm{id}_\mathfrak{g}$. Note that  a section of $p$ always exists.

Let $s : \mathfrak{g} \rightarrow \widehat{ \mathfrak{g}}$ be a section of the map $p$. We define a linear map $\rho : \mathfrak{g} \rightarrow \mathrm{End} (\mathcal{V})$ by $\rho (x) (u) : = [s(x), i(u)]_{\widehat{\mathfrak{g}}}$, for $x \in \mathfrak{g}$ and $u \in \mathcal{V}$. Then it can be easily check that $\mathcal{V} = (\mathcal{V}, \rho)$ is a representation of the Lie algebra $\mathfrak{g}$. More generally, $(\mathcal{V}, \mathcal{T})$ is a representation of the $\lambda$-weighted Rota-Baxter Lie algebra $(\mathfrak{g}, \mathfrak{T})$. One can easily check that this representation does'nt depend on the choice of the section $s$. We call this as the induced representation on $(\mathcal{V}, \mathcal{T})$.

Suppose $(\mathcal{V}, \mathcal{T})$ is a given representation of the $\lambda$-weighted Rota-Baxter Lie algebra $(\mathfrak{g}, \mathfrak{T})$. We denote by $\mathrm{Ext}(\mathfrak{g}, \mathcal{V})$  the isomorphism classes of abelian extensions of  $(\mathfrak{g}, \mathfrak{T})$ by $(\mathcal{V}, \mathcal{T})$ for which the induced representation on $(\mathcal{V}, \mathcal{T})$ is the prescribed one.

\begin{thm}
Let $(\widehat{\mathfrak{g}}, \widehat{\mathfrak{T}})$ be a $\lambda$-weighted Rota-Baxter Lie algebra and $(\mathcal{V}, \mathcal{T})$ be a representation of it. Then there is a one-to-one correspondence between $\mathrm{Ext}(\mathfrak{g}, \mathcal{V})$  and the second cohomology group $H^2_\mathsf{RB}(\mathfrak{g}, \mathcal{V}).$
\end{thm}

\begin{proof}
Let $(\psi, \chi) \in Z^2_\mathsf{RB} (\mathfrak{g}, \mathcal{V})$ be a $2$-cocycle, i.e., we have $\delta_\mathsf{CE} \psi = 0$ and $- \partial_\mathsf{CE}(\chi) - \Phi^2 (\psi) = 0$. Consider the space $\mathfrak{g} \oplus \mathcal{V}$ with the bracket
\begin{align*}
[(x,u), (y, v)] = ([x,y], \rho(x) v - \rho(y) u + \psi (x,y)).
\end{align*}
Since $\delta_\mathsf{CE} \psi =0$, it follows that the above bracket makes $\mathfrak{g} \oplus \mathcal{V}$ into a Lie algebra. We denote this Lie algebra by $\widehat{\mathfrak{g}}$. We also define a map $\widehat{\mathfrak{T}} : \widehat{\mathfrak{g}} \rightarrow \widehat{\mathfrak{g}}$ by $\widehat{\mathfrak{T}} (x,u) = (\mathfrak{T}(x), \mathcal{T}(u) + \chi (x))$, for $(x,u) \in \widehat{\mathfrak{g}}$. Since $- \partial_\mathsf{CE}(\chi) - \Phi^2 (\psi) = 0$, it follows that $\widehat{\mathfrak{T}}$ is a $\lambda$-weighted Rota-Baxter operator on the Lie algebra $\widehat{\mathfrak{g}}$. In other words, $(\widehat{\mathfrak{g}}, \widehat{\mathfrak{T}})$ is a $\lambda$-weighted Rota-Baxter Lie algebra. Moreover, the short exact sequence
\begin{align*}
\xymatrix{
0 \ar[r] & (\mathcal{V}, \mathcal{T}) \ar[r]^i & (\widehat{\mathfrak{g}}, \widehat{\mathfrak{T}}) \ar[r]^p & ({\mathfrak{g}}, {\mathfrak{T}}) \ar[r] & 0
}
\end{align*}
defines an abelian extension of $({\mathfrak{g}}, {\mathfrak{T}})$ by $(\mathcal{V}, \mathcal{T})$, where $i(u) = (0, u)$ and $p (x,u) = x$, for $x \in \mathfrak{g}$ and $u \in \mathcal{V}$. Note that the canonical section $s : \mathfrak{g} \rightarrow \widehat{ \mathfrak{g}}$, $s(x) = (x, 0)$ induces $(\mathcal{V}, \mathcal{T})$ with the original representation of the $\lambda$-weighted Rota-Baxter Lie algebra $(\mathfrak{g}, \mathfrak{T}).$

Let $(\psi', \chi') \in Z^2_\mathsf{RB} (\mathfrak{g}, \mathcal{V})$ be another $2$-cocycle cohomologous to $(\psi, \chi)$. Then there exists a pair $(\gamma, v) \in \mathrm{Hom} (\mathfrak{g},\mathcal{V}) \oplus \mathcal{V}$ such that
\begin{align*}
(\psi, \chi) - (\psi', \chi') = (\delta_\mathsf{CE} (\gamma),~ - \partial_\mathsf{CE}v - \Phi^1 (\gamma)).
\end{align*}
Consider the map $\phi : \mathfrak{g} \oplus \mathcal{V} \rightarrow \mathfrak{g} \oplus \mathcal{V}$ given by
\begin{align*}
\phi (x, u) = (x, u - \gamma (x) - \partial_\mathsf{CE}(v)(x) ).
\end{align*}
Then it can be easily checked that $\phi$ defines an isomorphism of abelian extensions from $(\widehat{\mathfrak{g}}, \widehat{\mathfrak{T}})$ to $(\widehat{\mathfrak{g}}', \widehat{\mathfrak{T}}')$. Therefore, there is a well-defined map $\Upsilon : H^2_\mathsf{RB}(\mathfrak{g}, \mathcal{V}) \rightarrow \mathrm{Hom}(\mathfrak{g}, \mathcal{V})$.

\medskip

Conversely, let $(\widehat{\mathfrak{g}}, \widehat{\mathfrak{T}})$ be an abelian extension given by (\ref{abelian}). Let $s : \mathfrak{g} \rightarrow \widehat{\mathfrak{g}}$ be a section of the map $p$. We define elements $\psi \in \mathrm{Hom}(\wedge^2 \mathfrak{g}, \mathcal{V})$ and $\chi \in \mathrm{Hom} (\mathfrak{g}, \mathcal{V})$ by
\begin{align*}
\psi (x, y) = [s(x), s(y)]_{  \widehat{\mathfrak{g}}} - s[x,y]  ~~~ \text{ and } ~~~ \chi (x) = \widehat{\mathfrak{T}} (s(x)) - s (\mathfrak{T}(x)), ~ \text{ for } x, y \in \mathfrak{g}.
\end{align*}
Then it follows from a straightforward computation that the pair $(\psi, \chi)$ defines a $2$-cocycle in $Z^2_\mathrm{RB} (\mathfrak{g}, \mathcal{V})$. Moreover, the corresponding cohomology class in $H^2_\mathrm{RB} (\mathfrak{g}, \mathcal{V})$ does'nt depend on the choice of the section $s$.

Let $(\widehat{\mathfrak{g}}, \widehat{\mathfrak{T}})$ and $(\widehat{\mathfrak{g}}', \widehat{\mathfrak{T}}')$ be two isomorphic abelian extensions as of (\ref{abelian-iso}). Let $s : \mathfrak{g} \rightarrow \widehat{\mathfrak{g}}$ be a section of the map $p$. Then we have
\begin{align*}
p' \circ ( \phi \circ s) = p \circ s = \mathrm{id}_\mathfrak{g}.
\end{align*}
This shows that $s' = \phi \circ s : \mathfrak{g} \rightarrow \widehat{\mathfrak{g}}'$ is a section of $p'$. If $(\psi', \chi')$ denotes the $2$-cocycle in $Z^2_\mathrm{RB} (\mathfrak{g}, \mathcal{V})$ corresponding to the abelian extension $(\widehat{\mathfrak{g}}', \widehat{\mathfrak{T}}')$  and section $s' : \mathfrak{g} \rightarrow \widehat{ \mathfrak{g}}'$ of the map $p'$, then
\begin{align*}
\psi' (x,y) =~& [s'(x), s'(y)]_{\widehat{\mathfrak{g}}'} - s' [x,y] \\
=~& [\phi \circ s (x), \phi \circ s (y)]_{\widehat{\mathfrak{g}}'}  - \phi \circ s [x,y] \\
=~& \phi \big(  [s(x), s(y)]_{\widehat{\mathfrak{g}}} - s[x,y]  \big) \\
=~& \phi (\psi (x,y)) = \phi (x,y)     \qquad (\because~ \phi|_\mathcal{V} = \mathrm{id}_\mathcal{V})
\end{align*}
and
\begin{align*}
\chi'(x) =~& \widehat{\mathfrak{T}}' (s'(x)) - s' (\mathfrak{T}(x)) \\
=~& \widehat{\mathfrak{T}}' (\phi \circ s(x)) - \phi \circ s (\mathfrak{T}(x)) \\
=~& \phi \big(  \widehat{\mathfrak{T}} (s(x)) - s (\mathfrak{T}(x))  \big) = \phi (\chi (x)) = \chi (x)   \qquad (\because~ \phi|_\mathcal{V} = \mathrm{id}_\mathcal{V}).
\end{align*}
This show that $(\psi, \chi)$ and $(\psi', \chi')$ correspond to the same elemwnt in $H^2_\mathrm{RB} (\mathfrak{g}, \mathcal{V})$. Therefore, we have a well-defined map $\Omega : \mathrm{Ext}(\mathfrak{g}, \mathcal{V}) \rightarrow H^2_\mathrm{RB} (\mathfrak{g}, \mathcal{V})$. Finally, the map $\Upsilon$ and $\Omega$ are inverses to each other. This completes the proof.
\end{proof}

\subsection{Formal deformations}
In this subsection, we will consider formal deformations of a $\lambda$-weighted Rota-Baxter Lie algebra $(\mathfrak{g}, \mathfrak{T})$. Here we will simultaneously deform the Lie bracket on $\mathfrak{g}$ and the $\lambda$-weighted Rota-Baxter operator $\mathfrak{T}$. We find the relation between such deformations and cohomology of $(\mathfrak{g}, \mathfrak{T})$ with coefficients in the adjoint representation.

Let $(\mathfrak{g}, \mathfrak{T})$ be a $\lambda$-weighted Rota-Baxter Lie algebra. Let $\mu \in C^2_\mathsf{CE} (\mathfrak{g}, \mathfrak{g}) = \mathrm{Hom}(\wedge^2 \mathfrak{g}, \mathfrak{g})$ be the element that corresponds to the Lie bracket on $\mathfrak{g}$, i.e., $\mu (x,y) = [x,y]$, for $x, y \in \mathfrak{g}$. Consider the space $\mathfrak{g}[[t]]$ of formal power series in $t$ with coefficients from $\mathfrak{g}$. Then $\mathfrak{g}[[t]]$ is a ${\bf k}[[t]]$-module.

\begin{defn}
A formal $1$-parameter deformation of $(\mathfrak{g}, \mathfrak{T})$ consists of a pair $(\mu_t, \mathfrak{T}_t)$ of two formal power series
\begin{align*}
&\mu_t = \sum_{i \geq 0} \mu_i t^i, ~ \text{ where } \mu_i \in \mathrm{Hom}(\wedge^2 \mathfrak{g}, \mathfrak{g}) \text{ with } \mu_0 = \mu,\\
&\mathfrak{T}_t = \sum_{i \geq 0} \mathfrak{T}_i t^i, ~ \text{ where } \mathfrak{T}_i \in \mathrm{Hom}(\mathfrak{g}, \mathfrak{g}) \text{ with } \mathfrak{T}_0 = \mathfrak{T}
\end{align*}
such that the ${\bf k}[[t]]$-bilinear map $\mu_t$ defines a Lie algebra structure on $\mathfrak{g}[[t]]$ and the ${\bf k}[[t]]$-linear map $\mathfrak{T}_t : \mathfrak{g}[[t]] \rightarrow \mathfrak{g}[[t]]$ is a $\lambda$-weighted Rota-Baxter operator. In other words, $(\mathfrak{g}[[t]] = (\mathfrak{g}[[t]], \mu_t), \mathfrak{T}_t)$ is a $\lambda$-weighted Rota-Baxter Lie algebra over ${\bf k}[[t]]$.
\end{defn}

Thus, $(\mu_t, \mathfrak{T}_t)$ is a formal $1$-parameter deformation of $(\mathfrak{g}, \mathfrak{T})$ if and only if
\begin{align*}
&\mu_t (x, \mu_t (y, z)) + \mu_t (y, \mu_t ( z, x)) + \mu_t (z, \mu_t (x, y)) = 0,\\
&\mu_t (  \mathfrak{T}_t (x), \mathfrak{T}_t(y)  ) = \mathfrak{T}_t \big(  \mu_t ( \mathfrak{T}_t(x),y  ) + \mu_t (x, \mathfrak{T}_t (y))  + \lambda \mu_t (x, y) \big),
\end{align*}
for $x,y,z \in \mathfrak{g}$. They are equivalent to the following two systems of identities: for $n \geq 0$ and $x,y, z\in \mathfrak{g}$,
\begin{align}
&\sum_{i+j = n} \mu_i (x, \mu_j (y, z)) + \mu_i (y, \mu_j ( z, x)) + \mu_i (z, \mu_j (x, y)) = 0, \label{sys-1}\\
&\sum_{i+j +k = n} \mu_i \big(   \mathfrak{T}_j (x) , \mathfrak{T}_k (y) \big) = \sum_{i+j+ k = n} \mathfrak{T}_i \big(  \mu_j (\mathfrak{T}_k (x) , y)  + \mu_j (x, \mathfrak{T}_k (y)) \big) + \sum_{i+j= n} \lambda~ \mathfrak{T}_i (\mu_j (x, y)). \label{sys-2}
\end{align}

\begin{defn}
Two formal deformations $(\mu_t, \mathfrak{T}_t)$ and $(\mu_t', \mathfrak{T}_t')$ of a $\lambda$-weighted Rota-Baxter Lie algebra $(\mathfrak{g}, \mathfrak{T})$ are said to be equivalent if there is a formal isomorphism
\begin{align*}
\phi_t = \sum_{i \geq 0} \phi_i t^i : \mathfrak{g}[[t]] \rightarrow \mathfrak{g} [[t]], ~ \text{ where } \phi_i \in \mathrm{Hom} (\mathfrak{g}, \mathfrak{g}) \text{ with } \phi_0 = \mathrm{id}_\mathfrak{g}
\end{align*}
such that the ${\bf k}[[t]]$-linear map $\phi_t$ is a morphism of $\lambda$-weighted Rota-Baxter Lie algebras from $(\mathfrak{g}[[t]]' =   (\mathfrak{g}[[t]], \mu_t'), \mathfrak{T}_t')$ to $(\mathfrak{g}[[t]] =   (\mathfrak{g}[[t]], \mu_t), \mathfrak{T}_t)$.
\end{defn}

Therefore, $(\mu_t, \mathfrak{T}_t)$ and $(\mu_t', \mathfrak{T}_t')$  are equivalent if the followings are hold:
\begin{align*}
\phi_t (\mu_t' (x,y) ) = \mu_t (\phi_t (x), \phi_t (y)) \quad \text{ and } \quad \phi_t \circ \mathfrak{T}_t' = \mathfrak{T}_t \circ \phi_t.
\end{align*}

They can be equivalently expressed by the following system of equations: for $n \geq 0$,
\begin{align}
\sum_{i+j = n} \phi_i  \big( \mu_j' (x,y)   \big) =~& \sum_{i+j+k = n} \mu_i \big( \phi_j (x), \phi_k (x) \big), \label{sys-3}\\
\sum_{i+j = n} \phi_i \circ \mathfrak{T}_j' (x) =~& \sum_{i+j = n} \mathfrak{T}_i \circ \phi_j (x). \label{sys-4}
\end{align}

\begin{thm}\label{2-co}
Let $(\mu_t, \mathfrak{T}_t)$ be a formal $1$-parameter deformation of the $\lambda$-weighted Rota-Baxter Lie algebra $(\mathfrak{g}, \mathfrak{T})$. Then $(\mu_1, \mathfrak{T}_1) \in C^2_\mathsf{RB} (\mathfrak{g}, \mathfrak{g})$ is a $2$-cocycle in the cohomology of $(\mathfrak{g}, \mathfrak{T})$ with coefficients in the adjoint representation. Moreover, the corresponding cohomology class in $H^2_\mathsf{RB} (\mathfrak{g}, \mathfrak{g})$ depends only on the equivalence class  of the deformation  $(\mu_t, \mathfrak{T}_t)$.
\end{thm}

\begin{proof}
Since $(\mu_t, \mathfrak{T}_t)$ is a formal $1$-parameter deformation, we get from (\ref{sys-1}) and (\ref{sys-2}) for $n=1$ that
\begin{align*}
[x, \mu_1 (y,z)] + \mu_1 (x, [y,z]) + [y, \mu_1 (z,x)] + \mu_1 (y, [z,x]) + [z, \mu_1 (x,y)] + \mu_1 (z, [x,y]) = 0
\end{align*}
and
\begin{align*}
\mu_1 & (\mathfrak{T}(x) , \mathfrak{T}(y)) + [ \mathfrak{T}_1 (x) , \mathfrak{T}(y)] + [ \mathfrak{T} (x) , \mathfrak{T}_1 (y)] \\
&= \mathfrak{T}_1 ([\mathfrak{T}(x), y] + [x, \mathfrak{T}(y)]) + \mathfrak{T} \big(  [\mathfrak{T}_1(x), y] + [x, \mathfrak{T}_1(y)] + \mu_1 (\mathfrak{T}(x) , y) + \mu_1 (x, \mathfrak{T}(y))  \big) \\
& ~~ + \lambda ~ \mathfrak{T} (\mu_1 (x,y)) + \lambda ~ \mathfrak{T}_1 ([x,y]).
\end{align*}
The first identity is equivalent to $\delta_\mathsf{CE} \mu_1 = 0$ while the second identity is equivalent to $- \partial_\mathsf{CE} (\mathfrak{T}_1) - \Phi^2 (\mu_1) = 0$. This implies that
\begin{align*}
\delta_\mathsf{RB} (\mu_1, \mathfrak{T}_1) = (\delta_\mathsf{CE} \mu_1,~ - \partial_\mathsf{CE} (\mathfrak{T}_1) - \Phi^2 (\mu_1) ) = 0.
\end{align*}
This proves the first part. For the second part, we let $(\mu_t, \mathfrak{T}_t)$ and $(\mu_t', \mathfrak{T}_t')$ be two equivalent deformations. For $n =1$, it follows from (\ref{sys-3}) and (\ref{sys-4}) that
\begin{align*}
(\mu_1' - \mu_1)(x,y) =~& [\phi_1(x), y] + [x, \phi_1(y)] - \phi_1 [x, y],\\
(\mathfrak{T}_1' - \mathfrak{T}_1)(x) =~& \mathfrak{T} ( \phi_1 (x)) -  \phi_1 (\mathfrak{T}(x)).
\end{align*}
Hence $(\mu_1', \mathfrak{T}_1') - (\mu_1, \mathfrak{T}_1) = (\delta_\mathsf{CE} \phi_1, - \Phi^1 (\phi_1)) = \delta_\mathsf{RB} (\phi_1, 0)$. This shows that $(\mu_1, \mathfrak{T}_1)$ and $(\mu_1', \mathfrak{T}_1')$ correspond to the same cohomology class in $H^2_\mathsf{RB}(\mathfrak{g}, \mathfrak{g})$. This completes the proof.
\end{proof}

We will now prove the following interesting result about formal $1$-parameter deformations.

\begin{thm}
Let $(\mathfrak{g}, \mathfrak{T})$ be a $\lambda$-weighted Rota-Baxter Lie algebra. If $H^2_\mathsf{RB} (\mathfrak{g}, \mathfrak{g}) = 0$ then any formal $1$-parameter deformation of $(\mathfrak{g}, \mathfrak{T})$ is equivalent to the trivial one $(\mu_t' = \mu, \mathfrak{T}_t' = \mathfrak{T}).$
\end{thm}

\begin{proof}
Let $(\mu_t, \mathfrak{T}_t)$ be a formal $1$-parameter deformation of the $\lambda$-weighted Rota-Baxter Lie algebra $(\mathfrak{g}, \mathfrak{T})$. It follows from Theorem \ref{2-co} that $(\mu_1, \mathfrak{T}_1)$ is a $2$-cocycle. Thus, from the hypothesis, there exists $(\psi_1, v) \in \mathrm{Hom}(\mathfrak{g}, \mathcal{V}) \oplus \mathcal{V}$ such that
\begin{align}\label{fin-def}
(\mu_1, \mathfrak{T}_1) = \delta_\mathsf{RB} (\psi_1, v) = (\delta_\mathsf{CE}(\psi_1), ~-\partial_\mathsf{CE} x - \Phi^1 (\psi_1)).
\end{align} 
Let $\phi_1 : \mathfrak{g} \rightarrow \mathfrak{g}$ be the map $\phi_1 = \mathrm{id}_\mathfrak{g} + \phi_1 t$. Then
\begin{align*}
\big( \overline{\mu}_t  = \phi_t^{-1} \circ \mu_t \circ (\phi_t \otimes \phi_t),~ \overline{\mathfrak{T}}_t = \phi_t^{-1} \circ \mathfrak{T}_t \circ \phi_t \big)
\end{align*}
is a deformation of $(\mathfrak{g}, \mathfrak{T})$ equivalent to $(\mu_t, \mathfrak{T}_t)$. Using (\ref{fin-def}), it can be easily check that the linear terms of $\overline{\mu}_t$ and $\overline{\mathfrak{T}}_t$ are vanish. In other words,  $\overline{\mu}_t$ and $\overline{\mathfrak{T}}_t$ are of the form
\begin{align*}
 \overline{\mu}_t = \mu + \overline{\mu}_2 t^2 + \cdots ~~ \text{ and } ~~ \overline{\mathfrak{T}}_t = \mathfrak{T} + \overline{\mathfrak{T}}_2 t^2 + \cdots .
\end{align*}
Hence by repeating this argument, we conclude that $(\mu_t, \mathfrak{T}_t)$ is equivalent to $(\mu_t' = \mu , \mathfrak{T}_t' = \mathfrak{T})$. Hence the proof.
\end{proof}

\begin{remark}
A $\lambda$-weighted Rota-Baxter Lie algebra $(\mathfrak{g}, \mathfrak{T})$ is said to be rigid if any deformation of it is equivalent to the trivial one. It follows from the above theorem that the vanishing of the second cohomology group $H^2_\mathsf{RB} (\mathfrak{g}, \mathfrak{g})$ is a sufficient condition for the rigidity of $(\mathfrak{g}, \mathfrak{T})$.
\end{remark}

\section{Rota-Baxter paired operators}\label{sec-rbp}


Let $\mathfrak{g}$ be a Lie algebra and $\mathcal{V} = (\mathcal{V}, \rho)$ be a representation of it. In this section, we introduce weighted Rota-Baxter paired operators that are related to weighted Rota-Baxter Lie algebras together with their representations. We introduce a differential graded Lie algebra whose Maurer-Cartan elements are weighted Rota-Baxter paired operators. Using such characterization, we define the cohomology of Rota-Baxter paired operators. Finally, we discuss deformations of weighted Rota-Baxter paired operators in terms of cohomology.

\begin{defn}
Let $\mathfrak{g}$ be a Lie algebra and $\mathcal{V} = (\mathcal{V}, \rho)$ be a representation of it. A pair $(\mathfrak{T}, \mathcal{T})$ of linear maps $\mathfrak{T} : \mathfrak{g} \rightarrow \mathfrak{g}$ and $\mathcal{T} : \mathcal{V} \rightarrow \mathcal{V}$ is said to be a $\lambda$-weighted Rota-Baxter paired operators if they satisfy the followings: for $x, y \in \mathfrak{g}$ and $u \in \mathcal{V}$,
\begin{align*}
[\mathcal{T}(x) , \mathfrak{T}(y)] =~& \mathfrak{T} \big( [\mathfrak{T}(x), y] + [ x, \mathfrak{T} (y)] + \lambda [x, y]   \big), \\
\rho (Tx) (\mathcal{T} u) =~& \mathcal{T} \big( \rho (\mathfrak{T}x) u + \rho (x) (\mathcal{T}u) + \lambda \rho (x) u  \big),
\end{align*}
\end{defn}

\begin{remark}
Observe that the first condition is equivalent to say that $(\mathfrak{g}, \mathfrak{T})$ is a $\lambda$-weighted Rota-Baxter Lie algebra, while the second condition is equivalent to say that $(\mathcal{V}, \mathcal{T})$ is a representation of $(\mathfrak{g}, \mathfrak{T})$. Therefore, $\lambda$-weighted Rota-Baxter paired operators are closely related with weighted Rota-Baxter Lie algebras and their representations. Hence examples of Rota-Baxter paired operators follow from the examples of Section \ref{sec-2}.
\end{remark}

\medskip

Let $\mathfrak{g}$ be a Lie algebra. Then the direct sum $\mathfrak{g} \oplus \mathfrak{g}$ carries a Lie algebra structure with bracket given by
\begin{align*}
[(x,x'), (y,y')]_\lambda = \big( [x,y],~ [x, y'] + [x', y] + \lambda [x', y']  \big).
\end{align*}
We denote this Lie algebra by $(\mathfrak{g} \oplus \mathfrak{g})_\lambda$. Moreover, if $\mathcal{V} = (\mathcal{V}, \rho)$ is a representation of the Lie algebra $\mathfrak{g}$, then $\mathcal{V} \oplus \mathcal{V}$ can be equipped with a representation of the Lie algebra $(\mathfrak{g} \oplus \mathfrak{g})_\lambda$ with the action given by
\begin{align*}
\rho_\lambda \big( (x,x') \big)(u,u')  = \big( \rho(x) u,~ \rho(x) u' + \rho (x') u + \lambda \rho (x') u'  \big).
\end{align*}
Hence the direct sum $\mathfrak{g} \oplus \mathfrak{g} \oplus \mathcal{V} \oplus \mathcal{V}$ carries the semidirect product Lie algebra structure whose bracket is explicitly given by
\begin{align*}
&\{ \!  [ (x,x', u, u'), (y, y', v, v') ] \! \}  \\
&= \big(   [x,y], ~ [x, y']+ [x', y]+ \lambda [x', y'], ~ \rho (x) v - \rho (y) u, ~ \rho (x) v' + \rho (x') v + \lambda (x') v' - \rho(y) u' - \rho (y') u - \rho (y') u' \big).
\end{align*}
We denote this Lie algebra by $(\mathfrak{g} \oplus \mathfrak{g} \oplus \mathcal{V} \oplus \mathcal{V})_\lambda$. The following proposition is straightforward, hence we omit the details.

\begin{prop}
Let $\mathfrak{g}$ be a Lie algebra and $\mathcal{V}$ be a representation of it. A pair $(\mathfrak{T}, \mathcal{T})$ of linear maps $\mathfrak{T} : \mathfrak{g} \rightarrow \mathfrak{g}$ and $\mathcal{T} : \mathcal{V} \rightarrow \mathcal{V}$ forms a $\lambda$-weighted Rota-Baxter paired operators if and only if
\begin{align*}
Gr_{\mathfrak{T}, \mathcal{T}} := \{ (\mathfrak{T}x, x, \mathcal{T}u, u)~|~ x \in \mathfrak{g}, ~ u \in \mathcal{V} \} \subset \mathfrak{g} \oplus \mathfrak{g} \oplus \mathcal{V} \oplus \mathcal{V}
\end{align*}
is a Lie subalgebra of $(\mathfrak{g} \oplus \mathfrak{g} \oplus \mathcal{V} \oplus \mathcal{V})_\lambda$.
\end{prop}


Let $\mathfrak{g}$ be a Lie algebra and $\mathcal{V}$ be a representation of it. We denote $\mathfrak{g}'$ and $\mathcal{V}'$ by another copies of $\mathfrak{g}$ and $\mathcal{V}$, respectively. Consider the vector space $\mathfrak{l} = \mathfrak{g} \oplus \mathfrak{g}' \oplus \mathcal{V} \oplus \mathcal{V}'$ and the Nijenhuis-Richardson bracket on the space $\oplus_{n \geq 1} \mathrm{Hom}(\wedge^n \mathfrak{l}, \mathfrak{l})$ of skew-symmetric multilinear maps on $\mathfrak{l}$, given by
\begin{align*}
[f, g]_\mathsf{NR} := f \diamond g - (-1)^{(m-1)(n-1)}~ g \diamond f ,
\end{align*} 
where
\begin{align*}
(f \diamond g) (l_1, \ldots, l_{m+n-1}) = \sum_{\sigma \in \mathbb{S}_{(n, m-1)}} (-1)^\sigma ~f\big( g (l_{\sigma(1)}, \ldots, l_{\sigma (n)}), l_{\sigma (n+1)}, \ldots, l_{\sigma (m+n-1)}  \big),
\end{align*}
for $f \in \mathrm{Hom}(\wedge^m \mathfrak{l}, \mathfrak{l})$ and $g \in \mathrm{Hom}(\wedge^n \mathfrak{l}, \mathfrak{l})$. The Nijenhuis-Richardson bracket $[~,~]_\mathsf{NR}$ makes the shifted graded vector space $\oplus_{n \geq 0} \mathrm{Hom}(\wedge^{n+1} \mathfrak{l}, \mathfrak{l})$ into a graded Lie algebra. Moreover, it can be easily checked that the graded subspace $\mathfrak{a} = \oplus_{n \geq 0} \big( \mathrm{Hom}(\wedge^{n+1} \mathfrak{g}', \mathfrak{g}) \oplus \mathrm{Hom}(\wedge^n \mathfrak{g}' \otimes \mathcal{V}', \mathcal{V} )  \big)$ is an abelian subalgebra.

In the subsequent, we will use the following notations. Let

\begin{itemize}
\item[$\diamondsuit$] $\mu \in \mathrm{Hom}(\wedge^2 \mathfrak{g}, \mathfrak{g})$ and $\mu' \in \mathrm{Hom}(\wedge^2 \mathfrak{g}', \mathfrak{g}')$ denote the multiplications of $\mathfrak{g}$ and $\mathfrak{g}'$, respectively;

\item[$\diamondsuit$] $\mathrm{ad}' : \mathfrak{g} \times \mathfrak{g}' \rightarrow \mathfrak{g}'$ be the adjoint representation of $\mathfrak{g}$ on $\mathfrak{g}'$;

\item[$\diamondsuit$] $\rho : \mathfrak{g} \times \mathcal{V} \rightarrow \mathcal{V}$ and $\rho' : \mathfrak{g} \times \mathcal{V}' \rightarrow \mathcal{V}'$ denote the representations of $\mathfrak{g}$ on $\mathcal{V}$ and $\mathcal{V}'$, respectively;

\item[$\diamondsuit$] $\varrho : \mathfrak{g}' \times \mathcal{V} \rightarrow \mathcal{V}$ and $\varrho' : \mathfrak{g}' \times \mathcal{V}' \rightarrow \mathcal{V}'$ denote the representations of $\mathfrak{g}'$ on $\mathcal{V}$ and $\mathcal{V}'$, respectively.
\end{itemize}

We consider the element
\begin{align*}
\theta = \mu + \mathrm{ad}' + \rho + \rho' + \varrho~ \in ~\mathrm{Hom}(\wedge^2 \mathfrak{l}, \mathfrak{l}).
\end{align*}
It is easy to see that $\theta$ is a Maurer-Cartan element in the graded Lie algebra $\big( \oplus_{n \geq 0} \mathrm{Hom}(\wedge^{n+1} \mathfrak{l}, \mathfrak{l}), [~,~]_\mathsf{NR}  \big)$.  Therefore, it induces a differential $d_\theta := [\theta, -]_\mathsf{NR}$ on the graded vector space $\oplus_{n \geq 0} \mathrm{Hom}(\wedge^{n+1} \mathfrak{l}, \mathfrak{l})$. Hence by the derived bracket construction of Voronov \cite{voro}, the shifted graded space 
\begin{align*}
\mathfrak{a}[-1] = \oplus_{n \geq 1} \big( \mathrm{Hom}(\wedge^{n} \mathfrak{g}', \mathfrak{g}) \oplus \mathrm{Hom}(\wedge^{n-1} \mathfrak{g}' \otimes \mathcal{V}', \mathcal{V} )  \big)
\end{align*}
inherits a graded Lie algebra structure with bracket given by
\begin{align}\label{derived-for}
\llbracket {\bf P}, {\bf Q} \rrbracket = (-1)^m ~[d_\theta ({\bf P}), {\bf Q} ]_\mathsf{NR},
\end{align}
for ${\bf P} = ({\bf P}_1, {\bf P}_2) \in  \mathrm{Hom}(\wedge^{m} \mathfrak{g}', \mathfrak{g}) \oplus \mathrm{Hom}(\wedge^{m-1} \mathfrak{g}' \otimes \mathcal{V}', \mathcal{V} )$ and ${\bf Q} = ({\bf Q}_1, {\bf Q}_2) \in \mathrm{Hom}(\wedge^{n} \mathfrak{g}', \mathfrak{g}) \oplus \mathrm{Hom}(\wedge^{n-1} \mathfrak{g}' \otimes \mathcal{V}', \mathcal{V} )$.

On the other hand, the element $\theta' = - (\mu' + \varrho') \in \mathrm{Hom}(\wedge^2 \mathfrak{l}, \mathfrak{l})$ is also a Maurer-cartan element in the graded Lie algebra $\big( \oplus_{n \geq 0} \mathrm{Hom}(\wedge^{n+1} \mathfrak{l}, \mathfrak{l}), [~,~]_\mathsf{NR}  \big)$. Hence it induces a differential $d_{\theta'} = [\theta', -]_\mathsf{NR}$ on $\oplus_{n \geq 1} \mathrm{Hom}(\wedge^{n} \mathfrak{l}, \mathfrak{l})$. The differential $d_{\theta'}$ restricts to a differential $\mathbbm{d}$ on the space $\mathfrak{a}[-1]$. Finally, the elements $\theta$ and $\theta'$ additionally satisfies $[\theta, \theta']= 0$ which in turn implies that $\mathbbm{d}$ is a derivation for the bracket $\llbracket ~,~\rrbracket$ on $\mathfrak{a}[-1]$. In other words, the triple $(\mathfrak{a}[-1], \llbracket ~, ~ \rrbracket, \mathbbm{d})$ is a differential graded Lie algebra.

\medskip

By replacing $\mathfrak{g}'$ by $\mathfrak{g}$ and replacing $\mathcal{V}'$ by $\mathcal{V}$, we get the following.

\begin{thm}
Let $\mathfrak{g}$ be a Lie algebra and $\mathcal{V}$ be a representation of it. Then the triple 
\begin{align*}
\big( \oplus_{n \geq 1} \big( \mathrm{Hom}(\wedge^{n} \mathfrak{g}, \mathfrak{g}) \oplus \mathrm{Hom}(\wedge^{n-1} \mathfrak{g} \otimes \mathcal{V}, \mathcal{V} )  \big), \llbracket ~, ~ \rrbracket, \mathbbm{d} \big)
\end{align*}
is a differential graded Lie algebra.  A pair ${\bf T} = (\mathfrak{T}, \mathcal{T})$ is a Rota-Baxter paired operators if and only if ${\bf T}$ is a Maurer-Cartan element in the above differential graded Lie algebra.
\end{thm}

\begin{proof}
The first part is clear from previous discussions. For any ${\bf T} = (\mathfrak{T}, \mathcal{T}) \in \mathrm{Hom}(\mathfrak{g}, \mathfrak{g}) \oplus \mathrm{Hom}(\mathcal{V}, \mathcal{V})$, by expanding (\ref{derived-for}) similar to \cite{tang}, we get that
\begin{align*}
\llbracket {\bf T}, {\bf T} \rrbracket_1 (x, y) =~& 2 \big( \mathfrak{T} ([\mathfrak{T}(x), y] + [x, \mathfrak{T}(y)]) - [\mathfrak{T}(x), \mathfrak{T}(y)]  \big),\\
\llbracket {\bf T}, {\bf T} \rrbracket_2 (x, u) =~& 2 \big( \mathcal{T} (   \rho (\mathfrak{T} x) u + \rho (x) (\mathcal{T} u) ) - \rho (\mathfrak{T}x)(\mathcal{T}u)   \big),
\end{align*}
for $x \in \mathfrak{g}$ and $u \in \mathcal{V}$. On the other hand, we can also check (similar to \cite{das-weighted}) that
\begin{align*}
(\mathbbm{d} {\bf T})_1 (x, y) = \lambda \mathfrak{T}[x, y] ~~~ \text{ and } ~~~ (\mathbbm{d} {\bf T})_2 (x, u) = \lambda \mathcal{T} (\rho(x) u).
\end{align*}
Hence it follows that $\mathbbm{d} {\bf T} + \frac{1}{2} \llbracket {\bf T}, {\bf T} \rrbracket = 0$ if and only if ${\bf T} = (\mathfrak{T}, \mathcal{T})$ is a Rota-Baxter paired operators.
\end{proof}

It follows from the above theorem that a Rota-Baxter paired operators ${\bf T} = (\mathfrak{T}, \mathcal{T})$ induces a differential
\begin{align*}
d_\mathbf{T} :  \mathrm{Hom}(\wedge^{n} \mathfrak{g}, \mathfrak{g}) \oplus \mathrm{Hom}(\wedge^{n-1} \mathfrak{g} \otimes \mathcal{V}, \mathcal{V} &\rightarrow \mathrm{Hom}(\wedge^{n+1} \mathfrak{g}, \mathfrak{g}) \oplus \mathrm{Hom}(\wedge^{n} \mathfrak{g} \otimes \mathcal{V}, \mathcal{V}, ~\text{ for } n \geq 1 ,
\end{align*}
\begin{align*}
d_\mathbf{T}& = \mathbbm{d} + \llbracket {\bf T}, - \rrbracket.
\end{align*}
We define $C^n_\mathsf{RBp}(\mathfrak{g}, \mathcal{V}) = \mathrm{Hom}(\wedge^{n} \mathfrak{g}, \mathfrak{g}) \oplus \mathrm{Hom}(\wedge^{n-1} \mathfrak{g} \otimes \mathcal{V}, \mathcal{V})$, for $n \geq 1$. Then $\{ C^*_\mathsf{RBp}(\mathfrak{g}, \mathcal{V}), d_\mathbf{T} \}$ is a cochain complex. The corresponding cohomology groups are called the cohomology of the Rota-Baxter paired operators $(\mathfrak{T}, \mathcal{T})$.

\medskip

Note that, for a Rota-Baxter paired operators ${\bf T} = (\mathfrak{T}, \mathcal{T})$, the differential $d_\mathbf{T}$ makes the tuple 
\begin{align*}
(\oplus_{n \geq 1} C^n_\mathsf{RBp}(\mathfrak{g}, \mathcal{V}), \llbracket ~,~ \rrbracket, d_\mathbf{T})
\end{align*}
into a differential graded Lie algebra. In the following, we will show that this differential graded Lie algebra governs the Maurer-Cartan deformations of ${\bf T}$.

\begin{prop}
Let $\mathfrak{g}$ be a Lie algebra and $\mathcal{V}$ be a representation of it. Let ${\bf T} = (\mathfrak{T}, \mathcal{T})$ be a Rota-Baxter paired operators. Then for a pair ${\bf T}' = (\mathfrak{T}', \mathcal{T}')$ of maps $\mathfrak{T}' : \mathfrak{g} \rightarrow \mathfrak{g}$ and $\mathcal{T}' : \mathcal{V} \rightarrow \mathcal{V}$, the sum $\mathbf{T} + \mathbf{T}' = (\mathfrak{T} + \mathfrak{T}', \mathcal{T} + \mathcal{T}')$ is a Rota-Baxter paired operators if and only if ${\bf T}'$ is a Maurer-Cartan element in the differential graded Lie algebra $(\oplus_{n \geq 1} C^n_\mathsf{RBp}(\mathfrak{g}, \mathcal{V}), \llbracket ~,~ \rrbracket, d_\mathbf{T})$.
\end{prop}

\begin{proof}
We observe that
\begin{align*}
&\mathbbm{d} ({\bf T} + {\bf T}') + \frac{1}{2} \llbracket {\bf T} + {\bf T}', {\bf T} + {\bf T}' \rrbracket \\
&= \mathbbm{d} {\bf T} + \mathbbm{d} {\bf T}' + \frac{1}{2} \big( \llbracket {\bf T}, {\bf T} \rrbracket + 2 \llbracket {\bf T}, {\bf T}' \rrbracket + \llbracket {\bf T}', {\bf T}' \rrbracket \big) \\
&= \mathbbm{d} {\bf T}' + \llbracket {\bf T}, {\bf T}' \rrbracket + \frac{1}{2} \llbracket {\bf T}', {\bf T}' \rrbracket = d_{\bf T} ({\bf T}') + \frac{1}{2} \llbracket {\bf T}', {\bf T}' \rrbracket.
\end{align*}
Hence the result follows from the Maurer-Cartan characterization of Rota-Baxter paired operators.
\end{proof}

\begin{remark}
In \cite{das-weighted} the present author studied formal $1$-parameter deformations of weighted Rota-Baxter operators by keeping the underlying algebra intact. Given a Lie algebra $\mathfrak{g}$ and a representation $\mathcal{V}$, one could also study formal $1$-parameter deformations of Rota-Baxter paired operators by keeping the algebra $\mathfrak{g}$ and the representation $\mathcal{V}$ intact. Then it is easy to see that the above-defined cohomology governs such deformations. More precisely, the infinitesimals lie in the first cohomology group, and the obstruction to extending a finite order deformation lies in the second cohomology group.
\end{remark}

\medskip

\noindent {\bf Acknowledgements.} The research is supported by the fellowship of Indian Institute of Technology (IIT) Kanpur. 

\medskip

\medskip


\end{document}